\documentclass[a4paper,11pt]{amsart}
\usepackage{amsmath,amsfonts,xcolor,amsthm,amssymb,mathrsfs,latexsym,paralist,soul,xypic}

\usepackage{tikz}
\usetikzlibrary{arrows,automata}

\tikzstyle{V}=[fill=black,circle,scale=0.4, outer sep = 4pt]

\numberwithin{equation}{section}

\newtheorem{thm}{Theorem}[section]
\newtheorem{prop}[thm]{Proposition}
\newtheorem{cor}[thm]{Corollary}
\newtheorem{lemma}[thm]{Lemma}

\theoremstyle{remark}
\newtheorem{rmk}[thm]{Remark}

\theoremstyle{definition}
\newtheorem{defn}[thm]{Definition}

\newcommand{\bi}{\begin{itemize}}
\newcommand{\ei}{\end{itemize}}
\newcommand{\be}{\begin{enumerate}}
\newcommand{\ee}{\end{enumerate}}

\newcommand{\C}{\mathbb{C}}

\newcommand{\T}{\mathbb{T}}

\newcommand{\G}{\mathcal{G}}

\newcommand{\R}{\mathbb{R}}
\newcommand{\N}{\mathbb{N}}
\newcommand{\Z}{\mathbb{Z}}

\newcommand{\Per}{\operatorname{Per}}

\date{18 July 2018}

\begin{document}

\title[Generalized gauge actions and Hausdorff structure]{Generalized gauge actions on $k$-graph $C^*$-algebras: KMS states and Hausdorff structure}
\author[Farsi, Gillaspy, Larsen and Packer]{Carla Farsi, Elizabeth Gillaspy, Nadia S. Larsen and Judith A. Packer}

\maketitle

\begin{abstract}
For a finite, strongly connected $k$-graph $\Lambda$, an Huef, Laca, Raeburn and Sims studied  the KMS states associated to the preferred dynamics of the $k$-graph $C^*$-algebra $C^*(\Lambda)$.  They found that these KMS states are determined
by the periodicity of $\Lambda$ and a certain Borel probability measure $M$ on the infinite path space $\Lambda^\infty$ of $\Lambda$. Here we consider different dynamics on $C^*(\Lambda)$, which arise from a functor $y: \Lambda \to \R_+$ and were first proposed by McNamara in his thesis.  
We show that the KMS states associated to McNamara's dynamics are again parametrized by the periodicity group of $\Lambda$ and a family of Borel probability measures on the infinite path space.  Indeed, these measures also arise as Hausdorff measures on $\Lambda^\infty$, and the associated Hausdorff dimension is intimately linked to the inverse temperatures at which KMS states exist.  Our construction of the metrics underlying the Hausdorff structure uses  the functors  $y: \Lambda \to \R_+$;
 the stationary $k$-Bratteli diagram associated to $\Lambda$; and the concept of exponentially self-similar weights on Bratteli diagrams.
\end{abstract}

%\tableofcontents
\section{Introduction}
 KMS states have their origin in equilibrium statistical mechanics and have long been a very fruitful tool in the study of operator algebras. In this paper, we identify links between KMS states on $C^*$-algebras of higher rank graphs, and the Hausdorff measure and Hausdorff dimension associated to  ultrametrics on   Bratteli diagrams that exhibit a certain self-similarity.

Given a $C^*$-algebra $A$ with a one-parameter group  of automorphisms $(\gamma_t)_{t\in \mathbb{R}}$, a state $\phi$ on $A$ satisfies the Kubo-Martin-Schwinger (KMS) condition at inverse temperature $\beta \in \mathbb{R}$ if $\phi(ab)=\phi(b\gamma_{i\beta}(a))$ for all analytic elements $a, b \in A$, where an element $x\in A$ is analytic if  the function $t\mapsto \gamma_t(x)$ extends to an entire function. Some fundamental examples of dynamical systems $(A,\gamma)$ with a unique KMS state at a distinguished inverse temperature $\beta$ are the  Cuntz algebras $\mathcal{O}_n$  with the gauge action, for $n\geq 1$, where $\beta=\log n$, see~\cite{olesen-pedersen}, and (when $A$ is an irreducible matrix) the Cuntz-Krieger  algebras $\mathcal{O}_A$ with the gauge action, where $\beta=\log \rho(A)$ with $\rho(A)$ the spectral radius of $A$, see ~\cite{enomoto-fujii-watatani}.  
This last result was generalized by Exel in \cite{exel}, and Exel and Laca in \cite{exel-laca-partial}, where they considered generalized gauge actions $\gamma$ on Cuntz-Krieger algebras.  Again, if $A$ is an irreducible matrix, the dynamical system $(\mathcal O_A, \gamma)$ admits a unique KMS state.  Links between these unique KMS states and harmonic measures were established by Okayasu in \cite{okayasu-first, okayasu}.
From a quite different perspective, it is natural to ask what are the possible inverse temperatures of dynamical systems, and in this respect Bratteli-Elliott-Herman~\cite{bratteli-elliott-herman} constructed simple $C^*$-algebras $A$ and associated dynamics $\gamma$  which attain any closed subset of $\mathbb{R}$ as a possible range of inverse temperatures for KMS$_\beta$ states.
  
Recently there has been great interest in KMS states for gauge actions on higher-rank graph $C^*$-algebras,  which are  a generalization of Cuntz-Krieger algebras: see for example~\cite{yang-endom, yang-von-Neumann, yang-factoriality, aHLRS-coord-irred, aHLRS, LLNSW, christensen}. The  first main aim of the present paper is the analysis, for the $C^*$-algebras of finite, strongly connected higher-rank graphs, of the KMS states associated to the generalized gauge actions on higher-rank graph $C^*$-algebras which were first introduced by McNamara in his thesis \cite{McNamara}.  Proposition \ref{prop:KMS-s-states} identifies the  inverse temperatures $\beta$ such that these actions admit KMS$_\beta$ states, while 
Theorem \ref{thm:KMS-simplex}  describes  the  KMS states.

Towards explaining this aim in more detail, let us first recall that Kumjian and Pask introduced higher-rank graphs (or $k$-graphs) and their $C^*$-algebras in \cite{kp} as a simultaneous generalization of the higher-rank Cuntz-Krieger algebras of Robertson and Steger \cite{robertson-steger} and the $C^*$-algebras of directed graphs.   In addition to their graph-theoretical description, the $C^*$-algebras associated to higher-rank graphs also admit a groupoid description as well as a universal presentation in terms of generators and relations.  This flexibility has led to applications of $k$-graph $C^*$-algebras in a variety of contexts (such as the question of nuclear dimension for Kirchberg algebras \cite{ruiz-sims-sorensen} and  $K$-theory computations for  quantum spheres \cite{hajac-sims}) and has also facilitated the analysis of structural properties of $k$-graph $C^*$-algebras.  For example, the ideal structure of $k$-graph $C^*$-algebras  $C^*(\Lambda)$ \cite{rsy1, robertson-sims, CKSS} is completely determined by the underlying higher-rank graph $\Lambda$, whereas the groupoid perspective enabled the characterization of Cartan subalgebras of $C^*(\Lambda)$  \cite{bnrsw}.

In the groupoid perspective, as explained by Renault already in  \cite{renault},  time evolutions (dynamics) on the $C^*$-algebra  of a groupoid $\mathcal{G}$ are implemented by continuous cocycles on $\mathcal{G}$,
  and the task of understanding the KMS states on $C^*(\mathcal{G})$ requires, at a minimum, identifying the quasi-invariant measures on the unit space of $\mathcal{G}$. There are now refinements of Renault's result, see for example \cite{neshveyev, thomsen, christensen}.  In particular, Christensen's recent preprint \cite{christensen} combines quasi-invariant measures with a certain group of symmetries to describe    KMS states on groupoid $C^*$-algebras.  This perspective is particularly well suited to our case of interest, namely, the KMS states associated to  generalized gauge actions on $k$-graph $C^*$-algebras.

Much of the structural analysis of $k$-graph $C^*$-algebras $C^*(\Lambda)$ is facilitated by the {\em gauge action}, a natural action of $\T^k$ on $C^*(\Lambda)$; the existing literature on KMS states for $k$-graph $C^*$-algebras is no exception.  Restricting the gauge action to a subgroup $R \cong \R$ of $\T^k$ gives rise to a dynamics, that is, a  one-parameter action $\alpha$ of the real line on $C^*(\Lambda)$. The KMS states and the possible range of inverse temperatures for the dynamical system  $(C^*(\Lambda), \alpha)$  carry interesting information about the underlying $k$-graph $\Lambda$, cf.~\cite{yang-von-Neumann,aHLRS-coord-irred, aHLRS,LLNSW,FGKP,yang-factoriality}. In particular, the analysis of \cite{aHLRS} links the KMS states associated to the gauge action with the simplicity of $C^*(\Lambda)$: for finite, strongly connected $k$-graphs $\Lambda$, simplicity of $C^*(\Lambda)$ is equivalent to the existence of a unique KMS$_\beta$ state at the inverse temperature $\beta=1$, and moreover to triviality of the  periodicity group $\Per\Lambda$ (which is defined in terms of shift invariant infinite paths in the $k$-graph; see Definition \ref{def:periodicity}).

For the method developed in \cite{aHLRS} to classify KMS states associated to dynamics of the form $\alpha$ described in the previous paragraph, a key ingredient is the construction of a certain Borel probability measure $M$ on the infinite path space $\Lambda^\infty$ of $\Lambda$. This measure $M$ is also intrinsically linked to both the fractal geometry and the noncommutative geometry of $\Lambda^\infty$ \cite{pearson-bellissard, FGJKP2}. In a different but related development, Ionescu and Kumjian established connections between KMS states and Hausdorff structure for certain Renault-Deaconu groupoid $C^*$-algebras in \cite{ionescu-kumjian-hausdorff-KMS}.  Their results apply in particular to the $C^*$-algebras associated to directed graphs (which are $k$-graphs for $k=1$) equipped with a generalized gauge dynamics.

 In this paper, we extend and sharpen these results for finite, strongly connected $k$-graphs, that is, $k$-graphs $\Lambda$ such that there are finitely many vertices in $\Lambda$, and the set $v\Lambda w$ of paths with source $w$ and range $v$ is finite and nonempty for each pair $(v, w)$ of vertices. We first analyze the KMS states on $C^*(\Lambda)$ for the generalized gauge dynamics $\alpha^{y,\theta}$ introduced by McNamara in \cite{McNamara}, where $\theta$ is a positive real number and $y$ is an {\em $\R_+$-functor}, or {\em weight functor} as  defined by McNamara in \cite{McNamara}.  Very simple examples involving a $2$-graph with a single vertex (cf.~Section \ref{sec:McNamara} below) show that there are multitudes of choices of $\R_+$-functors, and it is this flexibility we want to explore more closely.

McNamara's thesis \cite{McNamara}   characterizes the KMS states of $(C^*(\Lambda), \alpha^{y,\theta})$ under the additional hypothesis that each of the coordinate matrices $\{A_i\}_{i=1}^k$ of $\Lambda$ (defined in Equation \eqref{eq:coord-mcs}) is irreducible. This hypothesis implies, but is strictly stronger than, our standing hypothesis that $\Lambda$ be finite and strongly connected; see \cite[Lemma 4.1 and Example 4.3]{aHLRS}.  
The actions $\alpha^{y,\theta}$ are constructed using a new
family of matrices
$\{B_i(y,\theta)\}_{i=1}^k$
 that takes into account the coordinate matrices $A_1, \ldots, A_k$ as well as $y$ and $\theta$.  Lemma \ref{lem:B irred family} follows \cite[Proposition 3.1]{aHLRS} to establish that when $\Lambda$ is finite and strongly connected, the matrices $ B_i(y,\theta)$ admit a common positive eigenvector which is unique up to scaling; we call it the {\em Perron-Frobenius eigenvector} of the family $\{ B_i(y,\theta)\}_{i=1}^k$.
We prove in Theorem~\ref{thm:spec-rad-cts} that for a fixed  $y$, the spectral radii and the Perron-Frobenius eigenvector of the  family $B_1(y,\theta),\dots ,B_k(y,\theta)$ vary smoothly with $\theta$. This is a new sort of insight that would not have been available via considering only the gauge dynamics and its variations, and we  use it to prove that in some cases the unique inverse temperature $\beta>0$ at which $(C^*(\Lambda),\alpha^{y,\theta})$ admits KMS$_\beta$ states is precisely $\theta$, see Propositions~\ref{prop:one-vertex-unique-beta} and \ref{prop-partial-result-equality-thetas}.

 As hinted above, our analysis of the KMS states of $(C^*(\Lambda), \alpha^{y,\theta})$  comes in two steps: first we describe the unique quasi-invariant measure $\mu_{y,\theta}$ associated to $\alpha^{y,\theta}$ (Proposition \ref{prop:unique-measure} and Remark \ref{rmk:quasi-invar}), and then Theorem \ref{thm:KMS-simplex} shows that Christensen's group of symmetries for this dynamics is precisely $\Per\,\Lambda$.  We also show (see Proposition \ref{prop:KMS-s-states}) that the dynamical system $(C^*(\Lambda), \alpha^{y,\theta})$ admits KMS$_\beta$ states iff $\alpha^{y,\beta} = \alpha^{y,\theta}$.

One immediate consequence of  Theorem~\ref{thm:KMS-simplex} is a new proof of the  structural result from \cite{aHLRS} that identifies existence of a unique KMS$_\beta$ state with simplicity of $C^*(\Lambda)$ and aperiodicity of $\Lambda$. Moreover, if $\Lambda$ has irreducible coordinate matrices, we show that a certain technical condition (which arose already in  \cite{McNamara}) gives rise to a criterion for the aperiodicity of $\Lambda$ that may prove versatile in applications; see Corollary~\ref{cor:aperiodicity}.

Given the structural similarity between the KMS states associated to $\alpha^{y,\theta}$, and those which are tied to the actions $\alpha$ already studied by \cite{aHLRS}, we pause to reassure the reader that our added generality does indeed give rise to new examples of actions and measures.  To that end, in Section \ref{sec:examples} we study two higher-rank graphs which, while simple to draw, admit $\R_+$-functors $y$ leading to a diverse family of measures $\mu_{y,\theta}$ and actions $\alpha^{y,\theta}$. Their periodicity groups are also described, leading to a complete picture of the associated KMS states.
 
 Inspired by Ionescu and Kumjian \cite{ionescu-kumjian-hausdorff-KMS}, we then proceed to our second main aim of the paper; namely, we relate these KMS states to Hausdorff structures on $\Lambda^\infty$.  The same data of an $\R_+$-functor $y$ and a positive real number $\theta$ which gives rise to the generalized gauge actions $\alpha^{y,\theta}$ also leads to an ultrametric $d_{y,\theta}$ on  $\Lambda^\infty$ (Proposition \ref{prop:ultrametric}). 
  The construction of the ultrametric $d_{y,\theta}$ uses the stationary $k$-Bratteli diagram associated to $\Lambda$, which was introduced in \cite{FGJKP2},  as well as a new concept,  that of an  exponentially self-similar weight on a Bratteli diagram (Definition \ref{def-Cantorian-Brat-diagr}). 
  
   The final main result of the paper is Corollary \ref{cor:Hausdorff-dim}, which proves  that the Hausdorff dimension of $(\Lambda^\infty, d_{y,\theta})$ is $\theta$ and the associated Hausdorff measure is $\mu_{y,\theta}$.
   In fact, we establish a result about Hausdorff dimension in a greater generality involving weights on Bratteli diagrams with a certain self-similarity property; see Theorem \ref{thm:Hausdorff-dim-self-sim}.

\subsection*{Acknowledgments}
C.F.~and J.P.~were  partially supported by two individual  grants from the Simons Foundation (C.F.~\#523991; J.P.~\#316981).
E.G.~was partially supported by   the Deutsches Forschungsgemeinschaft via the SFB 878 ``Groups, Geometry, and Actions'' of the Universit\"at M\"unster.

The authors heartily thank Sooran Kang for a number of insightful conversations throughout the course of this work, and in particular for bringing McNamara's thesis \cite{McNamara} to our attention.
We also thank A.~Sims for encouraging us, after the second author's talk at the workshop ``Applications of operator algebras: order, disorder and symmetry" at ICMS Edinburgh during June 2017, to consider general finite, strongly connected $k$-graphs in our investigation.
Finally, thanks to J.~Christensen for alerting us to his recent preprint \cite{christensen}, which was instrumental in proving Theorem \ref{thm:KMS-simplex}.

The research presented here was completed in large part during visits of E.G.~and C.F.~to Oslo and of N.L.~to Boulder; we thank the host universities and their Operator Algebras research groups for their hospitality.

\section{Preliminaries on higher-rank graphs}\label{sec:prelim}

We begin by fixing the notation which will be used throughout the paper. For a matrix $A$, we write $\rho(A)$ for its spectral radius. By $\N$ we denote the monoid of natural numbers $\{0,1,2\dots\}$ under addition. For $k\in \N$, $k\geq 1$, we let $\N^k$ be the monoid of $k$-tuples of natural numbers under addition, with standard basis vectors $e_i$ for $i=1,\dots,k$.
 If $\gamma \in \R^k$, we will write $\gamma_i \in \R$ to denote the $i$th component of $\gamma$.  In other words, we have $\gamma = \sum_{i=1}^k \gamma_i e_i$.   If $\gamma$ is a vector in $\R^k$ and $m \in \N^k$, we write $\gamma^m$ for the product $\prod_{i=1}^k \gamma_i^{m_i}$.

 As a category, $\N^k$ has one object (namely 0), and composition of morphisms $n, m \in \N^k$ is given by addition.  In keeping with the use of $n\in \N^k$ to denote a morphism in the category $\N^k$, for any category $\Lambda$, we write
 \[ \lambda \in \Lambda \Leftrightarrow \lambda \in \text{Mor}(\Lambda).\]

We recall the basic facts of the construction, due to Kumjian and Pask, of $k$-graphs and their $C^*$-algebras \cite{kp}. Let $k\in\N$, $k\geq 1$. A $k$-graph $\Lambda$ consists of a countable small category and a \emph{degree functor} $d:\Lambda\to \N^k$, i.e.~a map such that $d(\eta\nu)=d(\eta)+d(\nu)$ for all $\eta,\nu\in \Lambda$, satisfying the following factorization property: whenever $d(\lambda)=m+n$ for $\lambda\in \Lambda$ and some $m,n\in \N^k$, there are unique $\eta,\nu\in \Lambda$ such that $\lambda=\eta\nu$, $d(\eta)=m$ and $d(\nu)=n$.

We denote by $\Lambda^n$  the set of morphisms (or paths) of degree $n\in\N^k$. Thus $\Lambda^0$ is the set of
identity morphisms; these are referred to as the vertices of $\Lambda$.
We also identify $\Lambda^0$ with the set of objects of $\Lambda$, and so the codomain
and domain maps become functions $r, s : \Lambda \to \Lambda^0$. In this paper, we
consider only finite $k$-graphs, meaning that $|\Lambda^n| < \infty$ for all $n\in \N^k$.
For $v, w \in \Lambda^0$ and $n \in \N^k$, denote by $v\Lambda^n w$ the set of $\lambda\in \Lambda^n$ such that  $s(\lambda)=w$ and $r(\lambda)=v$.

It is often helpful to think of $\Lambda^{e_i}$ as the ``edges of color $i$'' in $\Lambda$.  With this perspective, the factorization property implies that (for $i \not= j$) any morphism $\lambda \in \Lambda^{e_i + e_j}$ can be represented either as a color-$i$ edge followed by a color-$j$ edge, or a color-$j$ edge followed by a color-$i$ edge.  In other words, (cf.~\cite{hazle-raeburn-sims-webster}) a $k$-graph can be described by a directed graph with $k$ colors of edges, together with a pairing which identifies each color-$i$--color-$j$ path with a unique color-$j$--color-$i$ path.

The infinite path space $\Lambda^\infty$ of $\Lambda$ formally consists of degree-preserving functors from the $k$-graph $\Omega_k$  into $\Lambda$.  The category $\Omega_k$ has  object set $\N^k$ and morphism set
\[ \text{Mor}(\Omega_k) = \{(m,n) \in \N^k \times \N^k \mid m \le n\};\]
 its structure maps are given by $r(m,n) = m$, $s(m,n) = n$, and the composition rule is $(m,n)(n,p) = (m,p)$.  To view $\Omega_k$ as a $k$-graph, we equip it with the degree functor $d: \Omega_k \to \N^k$ given by
 $d(m,n) = n-m.$
The infinite path space  comes equipped with a family of {\em shift maps} $\{\sigma^j\mid j \in \N^k\}$, given by
\[ \sigma^j(x)(m,n) := x(m+j, n+j).\]

Observe that, since $\Omega_k$ has a terminal object (namely $0 \in \N^k$) but no initial object, our infinite paths will have a range but no source.  Moreover, the fact that $\Omega_k$ contains infinitely many morphisms of each degree implies that the same is true about every infinite path.  In particular, using the ``edge-colored directed graph'' perspective on $k$-graphs, every infinite path $x\in \Lambda^\infty$ contains infinitely many edges of each color.

 We  view $\Lambda^\infty$ as being topologized by the cylinder sets $\{Z(\lambda)\}_{\lambda \in \Lambda}$, where
\[ Z(\lambda) = \{ x \in \Lambda^\infty: x(0, d(\lambda)) = \lambda\}\]
is the collection of infinite paths with initial segment $\lambda$.
The topology on $\Lambda^\infty$ whose basic open sets are $\{ Z(\lambda)\}_{\lambda \in \Lambda}$ is locally compact and Hausdorff;  indeed, each cylinder set $Z(\lambda)$ is compact (as well as open) in the cylinder set topology.  If $\Lambda$ is finite then $\Lambda^\infty$ is compact.

The \emph{coordinate matrices} $A_1, \dots, A_k\in\operatorname{Mat}_{\Lambda^0}(\N)$ of $\Lambda$
are given by
\begin{equation}
\label{eq:coord-mcs}
A_i(v,w) = |v\Lambda^{e_i}w|
\end{equation}
for $v,w$ in $\Lambda^0$. By the factorization property, the matrices $A_i$ pairwise commute. For $n \in
\N^k$, we define
\begin{equation}\label{pairwisecomm}\textstyle
A^n := \prod^k_{i=1} A_i^{n_i}.
\end{equation}
Note that $A^n(v,w)=|v\Lambda^nw|$ for all $v,w$. The family $\{A_1,\dots ,A_k\}$ is an \emph{irreducible family of
matrices}, cf. \cite[Section 3]{aHLRS}, if each $A_i$ is nonzero and there is a finite set $F\subset \N^k$ such that the matrix
\[
A_F:=\sum_{n\in F} A^n
\] is positive; explicitly, for every pair $(v,w)\in \Lambda^0 \times \Lambda^0$ we have $A_F(v,w)>0$.

A  $k$-graph $\Lambda$ is said to be \emph{strongly connected} if $v\Lambda w \not= \emptyset$
for all $v,w \in \Lambda^0$. The following characterization then holds.

\begin{lemma}(\cite[Lemma 4.1]{aHLRS})\label{lemma: Lambda strongly connected} For a finite $k$-graph the following are equivalent:
\begin{enumerate}
\item[(i)] $\Lambda$ is strongly connected.
\item[(ii)] $\{A_1,\dots ,A_k\}$ is an irreducible family of matrices.
\item[(iii)] There exists a finite set $F\subset \N^k$ such that for all $v,w\in \Lambda^0$ there is
$\lambda\in \Lambda$ with $d(\lambda)\in F$ such that $s(\lambda)=w$ and $r(\lambda)=v$.
\end{enumerate}
\end{lemma}

The $k$-graph algebra $C^*(\Lambda)$ of a (finite) $k$-graph $\Lambda$ is the $C^*$-algebra that is universal for families $\{s_\lambda \mid
\lambda\in\Lambda\}$ satisfying
\begin{itemize}
\item[(CK1)] $\{s_v \mid v \in \Lambda^0\}$ is a family of mutually orthogonal
    projections;
\item[(CK2)] $s_\lambda s_\nu = s_{\lambda\nu}$ whenever $s(\lambda) = r(\nu)$;
\item[(CK3)] $s_\lambda^* s_\lambda = s_{s(\lambda)}$ for all $\lambda \in \Lambda$;
\item[(CK4)] For any $v\in \Lambda^0$ and any $n\in \N^k$, we have  $s_v = \sum_{\lambda \in v\Lambda^n} s_\lambda s^*_\lambda$.
\end{itemize}
Relations (CK4) and (CK3) enable us to rewrite any element $s_\alpha ^* s_\beta \in C^*(\Lambda)$ as a finite sum of terms of the form $ s_\lambda s_\eta^*$; to be precise, let
\[\Lambda^{min}(\alpha, \beta) := \{ (\xi, \zeta) \in \Lambda \times \Lambda \mid \alpha \xi = \beta \zeta \text{ and } d(\alpha \xi) = d(\alpha) \vee d(\beta)\}.\]
Here $d(\alpha) \vee d(\beta) $ denotes the coordinatewise maximum of $d(\alpha), d(\beta) \in \N^k$.
Moreover, relation (CK4) combines with the fact that each $s_\alpha$ is a  partial isometry to tell us that $\sum_{\rho: \alpha \not= \rho \in r(\alpha) \Lambda^{d(\alpha)}}  s_\rho s_\rho^* s_\alpha = 0$.  Consequently,
\begin{equation}
\label{eq:s-alpha s-beta}
\begin{split}
s_\alpha^* s_\beta &= \sum_{\xi \in s(\alpha) \Lambda^{d(\alpha) \vee d(\beta) - d(\alpha)}, \zeta \in s(\beta)\Lambda^{d(\alpha) \vee d(\beta) - d(\beta)}} s_\xi s_\xi^* s_\alpha^*   s_\beta s_\zeta s_\zeta^* \\
&= \sum_{(\xi, \zeta) \in \Lambda^{min}(\alpha, \beta)} s_\xi s_{ \alpha\xi}^* s_{\beta \zeta} s_\zeta^* = \sum_{(\xi, \zeta) \in \Lambda^{min}(\alpha, \beta)} s_\xi s_\zeta^*.
\end{split}
\end{equation}
In other words,
$\{ s_\lambda s_\eta^* \mid \lambda, \eta \in \Lambda\}$ densely spans $C^*(\Lambda)$.

We can also view $C^*(\Lambda)$  as a groupoid $C^*$-algebra \cite[Corollary 3.5(i)]{kp}.  Namely,  if we set
\[\G_\Lambda := \{ ( x, n, y ) \in \Lambda^\infty \times \Z^k \times \Lambda^\infty\mid \exists \ j, \ell \in \N^k \text{ s.t. } \sigma^j(x) = \sigma^\ell(y) \text{ and } n = j-\ell\},\]
then the isomorphism $C^*(\Lambda) \cong C^*(\G_\Lambda)$ identifies $s_\lambda s_\mu^* \in C^*(\Lambda)$ with the characteristic function $\chi_{Z(\lambda, \mu)} \in C_c(\G_\Lambda) \subseteq C^*(\G_\Lambda)$, where
\[ Z(\lambda, \mu) = \{ (x, d(\lambda) - d(\mu), y) \in \G_\Lambda \mid x \in Z(\lambda), \ y \in Z(\mu)\}.\]
The sets $\{Z(\lambda, \mu)\mid  s(\lambda) = s(\mu)\}$ constitute a compact open basis for the topology on $\G_\Lambda$, which is an ample groupoid.

Observe that $\Lambda^\infty \cong \{ (x, 0, x) \mid x \in \Lambda^\infty\}$ is the unit space of $\G_\Lambda$; therefore, $C(\Lambda^\infty)$ can be viewed as a subalgebra of $C^*(\G_\Lambda) \cong C^*(\Lambda)$.  The identification $\chi_{Z(\lambda, \mu)} \leftrightarrow s_\lambda s_\mu^*$, and the fact that the cylinder sets $Z(\lambda)$ are compact and open in $\Lambda^\infty$, allows one to see that $C(\Lambda^\infty) \cong \overline{\text{span}} \{ s_\lambda s_\lambda^* \mid \lambda \in \Lambda\}$.  Indeed, we have a canonical conditional expectation $\Psi: C^*(\Lambda) \to C(\Lambda^\infty)$ which is given on the generators by
$ \Psi(s_\lambda s_\mu^*) = \delta_{\lambda, \mu} s_\lambda s_\lambda^*.$

\section{$\R_+$-functors and measures on the infinite path space}\label{sec:measures}
The following definition is based on thinking of the non-negative real numbers as a category, with one object (namely 0) and composition of morphisms given by addition.  Throughout the paper, we will write 
\[ \R_+ := [0, \infty) \quad \text{ and } \quad \R_{>0} := (0, \infty).\]
\begin{defn}[\cite{McNamara}]
Let $\Lambda$ be a higher-rank graph.  A \emph{$
\mathbb R_+$-functor} (called a {\em weight functor} in \cite[Section 5.3]{McNamara}) on $\Lambda$ is a function $y: \Lambda \to [0, \infty)$ such that
\begin{itemize}
\item $y(v) = 0 $ for every $ v \in \Lambda^0$, and
\item $y(\lambda \nu) = y (\lambda) + y(\nu)$ whenever $s(\lambda) = r(\nu)$.
\end{itemize}
\end{defn}

We next recall another definition due to McNamara which constructs generalized coordinate matrices from a $\R_+$-functor and a nonnegative parameter.

\begin{defn}(\cite[Definition 5.10]{McNamara}) Let $\Lambda$ be a finite $k$-graph. For $y:\Lambda\to [0,\infty)$  a $\mathbb R_+$-functor and $\theta\in [0,\infty)$, define matrices $B_i(y,\theta)$ for each $i=1,\dots ,k$ by
\begin{equation}\label{eq:Bi matrices}
B_i(y,\theta)_{v,w}: = \sum_{\lambda \in v\Lambda^{e_i} w} e^{-\theta y(\lambda)}.
\end{equation}
\label{def:Bi-mcs}
\end{defn}

It is established in \cite[Lemma 5.11]{McNamara} that for any choice of $y$ and $\theta$, the matrices $\{B_i(y,\theta)\}_{i=1, \ldots, k}$  pairwise commute. Further,  in \cite[Lemma 5.13]{McNamara} McNamara proves that whenever all of the matrices $A_1,\dots ,A_k$ are irreducible, so are all the matrices $B_i(y,\theta)$.

We next notice that if $\Lambda$ is a finite strongly connected graph, then for every choice of $\R_+$-functor $y$ and $\theta$ in $[0,\infty)$, the matrices $\{B_i(y,\theta)\}_{i=1,\dots, k}$ form an irreducible family.

\begin{lemma}\label{lem:B irred family} Let $\Lambda$ be a strongly connected $k$-graph,  $y:\Lambda\to [0,\infty)$ a $\R_+$-functor and $\theta \in [0,\infty)$. Then $\{B_i(y,\theta)\}_{i=1,\dots, k}$ is an irreducible family of matrices.
\end{lemma}

\begin{proof}
For each $n\in \N^k$  we have (cf.~\cite[Definition 5.12]{McNamara}) a new $\Lambda^0\times \Lambda^0$ matrix:
\begin{equation}\label{eq:B upper multiindex n}
B(y,\theta)^n:=\prod_{i=1}^k B_i(y,\theta)^{n_i}.
\end{equation}
The fact that the matrices $B_i(y,\theta)$ pairwise commute implies that this definition is independent of our choice of ordering on the generators of $\N^k.$

Since $y$ is additive, the $(v,w)$ entry in $B(y,\theta)^n$ is given by
\begin{equation}\label{eq: B upper multiindex n at vw}
B(y,\theta)^n_{v,w}=\sum_{\lambda\in v\Lambda^nw}e^{-\theta y(\lambda)}.
\end{equation}
Let $F\subset \N^k$ be the finite set of degrees given by Lemma~\ref{lemma: Lambda strongly connected}. Similar to the definition of $A_F$, but using an upper subscript to ease the notation, let
\[B(y,\theta)^F:=\sum_{n\in F} B(y,\theta)^n.
\]
It suffices to prove that $B(y,\theta)^F$ is positive, that is, for each $v,w$ in $\Lambda^0$ we have $B(y,\theta)^F_{v,w}>0$. By construction,
\[
B(y,\theta)^F_{v,w}=\sum_{n\in F}\sum_{\lambda\in v\Lambda^n w}e^{-\theta y(\lambda)}
\]
and by our assumption, $\sum_{n\in F}\vert v\Lambda^n w\vert = A_F(v,w) \not= 0$.
\end{proof}

Consequently, \cite[Proposition 3.1]{aHLRS} implies that there is a unique vector $\xi^{y,\theta} \in \left( \R_{>0} \right)^{\Lambda^0}$ with $\sum_{v\in \Lambda^0} \xi^{y,\theta}_v = 1$ and
\[ B_i(y,\theta) \xi^{y,\theta} = \rho(B_i(y,\theta)) \xi^{y,\theta} \ \text{ for all }   1\leq i \leq k,\]
where $\rho(B_i(y,\theta))>0$ is the spectral radius of the matrix $B_i(y,\theta)$.

\begin{defn}[Notation]
Throughout this paper, we will write 
\[ \rho(B(y,\theta)) := (\rho (B_1(y,\theta)), \ldots, \rho (B_k(y,\theta))).\]
For $m\in \Z^k$, we let $\rho(B(y,\theta))^m $ denote the product $\prod_{i=1}^k\rho (B_i(y,\theta))^{m_i}$.
\end{defn}

Measures on the infinite path space of a finite, strongly connected $k$-graph have recently gained attention starting with the construction in \cite[Proposition 8.1]{aHLRS} that was motivated by the study of KMS states, and continued by constructions relating to Hausdorff and Markov measures in  \cite{FGJKP1,FGJorKP}. Here we construct a Borel probability measure on $\Lambda^\infty$ for any $\R_+$-functor and parameter $\theta\in [0,\infty)$.
\begin{prop}
\label{prop:unique-measure}
Let $\Lambda$ be a finite, strongly connected $k$-graph, $y$ a $\R_+$-functor on $\Lambda$, and $\theta \in [0,\infty)$.
The measure $\mu_{y,\theta}$ on $\Lambda^\infty$ given by
\begin{equation}
\mu_{y,\theta}(Z(\lambda)) =  e^{-\theta y(\lambda)} \rho(B(y,\theta))^{-d(\lambda)}\xi^{y,\theta}_{s(\lambda)}\label{eq:mu-y-beta}
\end{equation}
is the unique measure $\mu$ on $\Lambda^\infty$ such that $\mu (Z(v)) > 0$ for all $v \in \Lambda^0$ and
\begin{equation}
\mu(Z(\lambda)) = e^{-\theta y(\lambda)} \rho(B(y,\theta))^{-d(\lambda)} \mu(Z(s(\lambda)).
\label{eq:conformal-condition}
\end{equation}
Moreover, the measures $\mu_{y,\theta}$ are all probability measures.
\end{prop}

\begin{proof}
The fact that $y(v) = 0$ whenever $v \in \Lambda^0$ implies that $ \mu_{y,\theta}(Z(v)) = \xi^{y,\theta}_v > 0$ for all $v \in \Lambda^0$, and that $ \mu_{y,\theta}(Z(v)) $  satisfies Equation \eqref{eq:conformal-condition}.  Since
\[ \mu_{y,\theta}(\Lambda^\infty) = \sum_{v \in \Lambda^0} \mu_{y,\theta}(Z(v)) = \sum_{v \in \Lambda^0} {\xi}^{y,\theta}_v = 1,\]
showing that $\mu_{y,\theta}$ is indeed a measure will imply that $\mu_{y,\theta}$ is a probability measure.

To establish the claim that $\mu_{y,\theta}$ is a measure, we first observe that it is finitely additive on  cylinder sets $Z(\lambda)$ with $d(\lambda) = (n, \ldots, n)\in \N^k$ for some $n\in \N$.  To see this, it suffices to show that
\[\mu_{y,\theta}(Z(v)) = \sum_{\lambda \in v\Lambda^{(n, \ldots, n)}} \mu_{y,\theta}(Z(\lambda))\]
for any $n \in \N$.  Observe that
\begin{align*}
\sum_{\lambda \in v\Lambda^{(n, \ldots, n)}} \mu_{y,\theta}(Z(\lambda)) & = \rho(B(y,\theta))^{-(n,\ldots, n)} \sum_{w\in \Lambda^0} B(y,\theta)^{(n, \ldots, n)}(v,w) \xi^{y,\theta}_w \\
&  = \xi^{y,\theta}_v = \mu(Z(v))\end{align*}
by the definition of the matrices $B_i(y,\theta)$ and their common eigenvector $\xi^{y,\theta}$.  In other words, $\mu_{y,\theta}$ is indeed finitely additive on square cylinder sets.

The fact that $\mu_{y,\theta}$ defines a measure on $\Lambda^\infty$ now follows from the Kolmogorov Extension Theorem  (cf.~Lemma 2.12
of \cite{FGJorKP}).

 Suppose now that a measure $\mu$ satisfies  Equation
\eqref{eq:conformal-condition}.  Without loss of generality, we will further assume that $\mu(\Lambda^\infty) = 1$.  Define $m \in \R_+^{\Lambda^0}$ by $m(v) = \mu(Z(v))$.  For any $1 \leq i \leq k$, Equation \eqref{eq:conformal-condition} implies that
\begin{align*}
m(v) &= \sum_{\lambda \in v\Lambda^{e_i}}\mu(Z(\lambda)) = \sum_{\lambda \in v\Lambda^{e_i}} e^{-\theta y(\lambda)} \rho(B_i(y,\theta))^{-1} m(s(\lambda)) \\
&= \rho(B_i(y,\theta))^{-1} \sum_{w \in \Lambda^0} B_i(y,\theta)_{v,w} m(w).
\end{align*}
In other words, $m \in \R^{\Lambda^0}_+$ is a positive eigenvector for each matrix $B_i(y,\theta)$, with eigenvalue $\rho(B_i(y,\theta))$.  Moreover, our hypothesis that $\mu(\Lambda^\infty) = 1$ implies that $\sum_{v\in \Lambda^0} m(v) = 1$.  Proposition 3.1 of \cite{aHLRS} therefore implies that $m = \xi^{y,\theta}$, so
\[ \mu(Z(v)) = \mu_{y,\theta}(Z(v))\]
for all $v\in \Lambda^0$.  Equation \eqref{eq:conformal-condition} now implies that $\mu(Z(\lambda)) = \mu_{y,\theta}(Z(\lambda))$ for all $\lambda \in \Lambda$.  Since the topology (and the associated Borel structure) of $\Lambda^\infty$ are generated by the cylinder sets $Z(\lambda)$, it follows that
$\mu = \mu_{y,\theta}.$
\end{proof}

\begin{thm}
\label{thm:spec-rad-cts}
Let $\Lambda$ be a finite, strongly connected $k$-graph, $y$ a $\R_+$-functor on $\Lambda$, and $\theta\in[0,\infty)$. The common eigenvector $\xi^{y,\theta}$ for the matrices $B_i(y,\theta)$, as well as
 the spectral radii $\{\rho(B_i(y,\theta))\}_{i=1}^k,$ depend smoothly on the entries of the matrix $B_i(y,\theta)$.  In particular, the spectral radii   and the eigenvector $\xi^{y,\theta}$ vary smoothly with $\theta$.
\end{thm}

\begin{proof}
Inspired by \cite{yeo-blog}, we use the Implicit Function Theorem. Suppose that $\Lambda$  has $n$ vertices.
Since $\Lambda$ is strongly connected, we know from Lemma \ref{lem:B irred family} that there is a finite subset $ F \subseteq \N^k$ (which does not depend on $y$ or $\theta$) such that $B(y,\theta)^{ F}$ is a positive matrix.  Moreover, $\xi^{y,\theta}$ is the 
unique eigenvector for $B(y,\theta)^F$ with eigenvalue
$\rho(B(y,\theta)^F)$ and $\ell^1$-norm 1.

Observe first that there must exist at least one $i, \ 1 \leq i \leq k$, such that for some $f \in F$ we have $f_i \not= 0$.  (If this assertion were false then $B(y,\theta)^F = I$ would be the identity matrix, which is not positive.)  By re-ordering the indices $i$ if necessary, we will assume that
\begin{equation}
\exists \ f \in F : f_1 \not= 0 .
\label{eq:f_1-nonzero}
\end{equation}

To a vector $\vec x$  in $\R^{kn^2}$,
\[ \vec x =(\vec{x}^1, \ldots, \vec{x}^k), \quad \text{ where } \quad \vec{x}^\ell = (x^\ell_{11}, \ldots, x^\ell_{1n}, x^\ell_{21}, \ldots,x^\ell_{2n},  \ldots, x^\ell_{n1}, \ldots,  x^\ell_{nn}),\]
we associate the matrices $X_\ell := (x^\ell_{ij})_{i,j} \in M_n(\R)$, for $1 \leq \ell \leq k$. Then, for a family $X_1, \ldots, X_k$ of matrices, we  define
\[ X^F = \sum_{f \in F} \prod_{i=1}^k X_i^{f_i}.\]
  Similarly, if $\vec p = (p_1, \ldots, p_k)$, we define
$  p^F := \sum_{f\in F} \prod_{i=1}^k p_i ^{f_i}.$

Let $\vec v= (v_1, \ldots, v_n)$. We now define a smooth function $f: \R^{kn^2+ k + n}\to \R^{k+n}$ by
\begin{align*} f(& \vec{x}, \vec p, \vec v)\\
&
 =\left(-1 + \sum_{i=1}^n v_i , [X_2(\vec v)]_1 -  p_2v_1, \ldots,[X_k(\vec{v})]_1 - p_k v_1,  (X^F -  p^F\cdot I)\vec{v}\right).
\end{align*}
If  $\vec v$ is an eigenvector for each matrix $X_\ell = (x^\ell_{ij})_{i,j}$ with eigenvalue $p_\ell$, and $\| \vec v \|_1 = 1$, then $f(x^1_{11}, \ldots, x^k_{nn}, \vec p, \vec v) = \vec 0$. Moreover, the $(k+n) \times (k+n)$ matrix of partial derivatives of $f$ with respect to the variables $p_1, \ldots, p_k, v_1, \ldots, v_n$ is given by $\mathcal F = (F_{i,j})_{i,j}$, where if $2 \leq i, j \leq k$ and $1 \leq h \leq n$, we have
\[ F_{i,j} = \frac{\partial}{\partial p_j}(  [X_i(\vec v)]_1 - p_i v_1) = - \delta_{i,j} v_1; \qquad F_{i, k+h}= \frac{\partial}{\partial v_{h}}(  [X_i(\vec v)]_1 - p_i v_1) = x^i_{1h} .\]
Note that $F_{1, j} = \frac{\partial}{\partial p_j} \left( \sum_{h =1}^n v_h \right) = 0$ for all $1 \leq j \leq k$. Similarly,
\[F_{i, 1} = \frac{\partial}{\partial p_1} (  [X_i(\vec v)]_1 - p_i v_1) = 0 \ \forall \ 1 \leq i \leq k. \]
Moreover, $F_{1, k+h} = \frac{\partial}{\partial v_h} \left( \sum_{m=1}^n v_m - 1\right) = 1$.  If $1 \leq i \leq k$ and $1 \leq h\leq n$,
\[ F_{k+ h, i} = \frac{\partial}{\partial p_i}  \left( - p^F v_h +\sum_{j=1}^n ( X^F)_{h,j} v_j \right) = - v_h \sum_{f\in F} f_i p_i^{f_i -1} \prod_{j\not= i } p_j^{f_j} .\]
Finally,
if $1 \leq i , j \leq n $ we have
\[ F_{k+i, k+j}  = \frac{\partial}{\partial v_j} \left( - p^F v_i +\sum_{j=1}^n ( X^F)_{i,j} v_j \right) = X^F_{i,j} - \delta_{i,j} p^F.\]

Applied to the  situation when $X_i = B_i(y,\theta)$, $\vec{v} = \xi^{y,\theta}$, and $p_i= \rho(B_i(y,\theta)) = : \rho_i,$ the Implicit Function Theorem tells us that whenever the $(k+n) \times (k+n)$ matrix
\[ \mathcal F = \begin{pmatrix}
0 & 0 \qquad \cdots \qquad  0 & 1 \qquad  \cdots \qquad  1 \\
\vec{0} &
\text{diag}\, (- \xi^{y,\theta}_1, \ldots, - \xi^{y,\theta}_1) & \left ( B_i(y,\theta)_{1, h}\right)_{2 \leq i \leq k, \, 1 \leq h\leq n}  \\
& \left( -\xi^{y,\theta}_h \sum_{f\in F} \frac{f_i}{\rho_i} \prod_{j=1}^k \rho_j f_j \right)_{h, i} & (B(y,\theta)^F - \rho^F \cdot I)
\end{pmatrix}
\]
is invertible, the eigenvalues $\rho(B_i(y,\theta))$ and the common eigenvector $\xi^{y,\theta}$ of $\{ B_i(y,\theta)\}_{i=1}^k$ depend smoothly on the entries of the matrices $B_i(y,\theta)$.

  We now proceed to show that $\mathcal F$ is invertible.
Suppose that $\mathcal F\begin{pmatrix}
 q \\
\zeta
\end{pmatrix}=0$.  Then $\| \zeta\|_1 = 0$.  Moreover,
 for all $2 \leq i \leq k$,
\begin{equation}
\label{eq:if-F-q-zeta-equals-zero-1}
 [B_i(y,\theta)(\zeta)]_1 =  q_i \xi^{y,\theta}_1  \end{equation}
\[\text{ and }  \left( -\xi^{y,\theta}_h\sum_{i=1}^k \sum_{f\in F} \frac{f_i q_i }{\rho_i} \prod_{j=1}^k \rho_j f_j \right)= \sum_{m=1}^n B^F_{h, m} \zeta_m  -\rho^F \zeta_h, \ \text{equivalently,}\]
\begin{equation}
\label{eq:if-F-q-zeta-equals-zero}
\begin{split}
&
  (B^F - \rho^F \cdot I )(\zeta) = \left(\sum_{i=1}^k \sum_{f\in F} \frac{f_i q_i }{p_i} \prod_{j=1}^k p_j f_j \right) \xi^{y,\theta}.
  \end{split}\end{equation}
In particular, Equation \eqref{eq:if-F-q-zeta-equals-zero} implies that $(B^F - \rho^F \cdot I)(\zeta)$ is a multiple of $\xi^{y,\theta}$.

 Let $J$ be the Jordan form of $B^F$, and let  $\mathscr S:=  \{ \xi^{y,\theta},  x_1, \ldots,  x_{n-1}\}$ be a corresponding basis of $\R^n$.
Observe that  $ [ B^F - \rho^F \cdot I]_\mathscr S$ is an upper triangular matrix. If $\rho^F= \alpha_0, \alpha_1, \ldots, \alpha_{n-1}$ are the eigenvalues of $B^F,$ counted with multiplicity, then the diagonal of $ [ B^F - \rho^F \cdot I]_\mathscr S$ is given by $0, \alpha_1 - \rho^F, \ldots, \alpha_{n-1}- \rho^F$.

Writing $\zeta = z_0 \xi^{y,\theta} + \sum_{i=1}^{n-1} z_i x_i$, the fact that $[B^F - \rho^F \cdot I]_\mathscr S$ is upper triangular and has first row  zero implies that $(B^F - \rho^F \cdot I) (\zeta)$ is a linear combination of the vectors $(x_i)_{i=1}^{n-1}$:
\[[B^F - \rho^F \cdot I]_\mathscr S (\zeta) = \sum_{i=1}^{n-1} [ (\alpha_i - \rho^F) z_i  + J_{i, i+1} z_{i+1} ]x_i, \]
where we add $z_{i+1}$ only for those indices $i$ for which $J_{i, i+1} = 1$.

Since $(B^F - \rho^F \cdot I)(\zeta)$ is a multiple of $\xi^{y,\theta}$ whenever  $\mathcal F\begin{pmatrix}
 q \\ \zeta
\end{pmatrix} = 0$, the fact that $\alpha_{n-1} \not= \rho^F$ implies that $z_{n-1} = 0$ in this case.  Proceeding  ``backwards'' by induction (from $n-1$ towards $1$) reveals that, in fact, $z_i =0$ for all $1\leq i \leq n-1$.  In other words,
\[\mathcal F\begin{pmatrix}
 q \\ \zeta
\end{pmatrix} = 0 \Rightarrow \zeta = z_0 \xi^{y,\theta}.\]In particular, $B_i(y,\theta)(\zeta) = z_0 \rho_i \xi^{y,\theta}$ and (computed in the original basis) $\| \zeta\|_1 = z_0$.
Thus, the fact that $\|\zeta\|_1 = 0$ implies $z_0 = 0$, and hence $\zeta = 0$.

Moreover, using these formulas in Equation \eqref{eq:if-F-q-zeta-equals-zero-1} implies that
\[  q_i = z_0 \rho_i  \ \forall \ 2 \leq i \leq k \quad  \text{and hence}  \quad  q_i = 0 \text{ if } i > 1.\]

Recall from Equation \eqref{eq:if-F-q-zeta-equals-zero}  that $z_0 = \sum_{i=1}^k \sum_{f\in F} \frac{f_i q_i }{\rho_i} \prod_{j=1}^k \rho_j f_j $.
It then follows from the previous equation that
\[ q_1 \sum_{f\in F} \frac{f_1}{\rho_1}\prod_{j=1}^k \rho_j f_j  = 0.\]
Every term in the sum is non-negative, since $\rho_j > 0$ for all $j$ and $f_j \geq 0$.  Moreover, Equation \eqref{eq:f_1-nonzero} guarantees that there is at least one nonzero term in the sum; consequently, $q_1 = 0$.  We conclude that if $\mathcal F\begin{pmatrix}
q \\ \zeta
\end{pmatrix} = 0$ then $q = 0$ and $\zeta = 0$, as claimed.
\end{proof}

\begin{cor} Let $\Lambda$ be a finite, strongly connected $k$-graph equipped with a $\R_+$-functor $y$. In the weak*-topology on the space of measures on $\Lambda^\infty$, the function $\theta \mapsto \mu_{y,\theta}$ is continuous on $\R_{+}$.
\label{cor:continuity-of-measures}
\end{cor}
\begin{proof}
 Theorem \ref{thm:spec-rad-cts} and the definition \eqref{eq:mu-y-beta} of $\mu_{y,\theta}$ combine to tell us that for each fixed $\lambda \in \Lambda$, the function
\[ \theta \mapsto \mu_{y,\theta}(Z(\lambda))\]
is continuous (in fact smooth).  Since $\{ \chi_{Z(\lambda)}:\lambda \in \Lambda\}$ densely spans $C(\Lambda^\infty)$, whose dual is the space of measures on $\Lambda^\infty$, standard measure-theoretic arguments enable us to complete the proof.
\end{proof}

We conclude this section with some remarks about the relationship between  the measures $\mu_{y,\theta}$ and the periodicity of $\Lambda$.

\begin{defn}\cite[Lemma 3.2]{robertson-sims}
\label{def:periodicity}
A $k$-graph $\Lambda$ has {\em periodicity at }$v \in \Lambda^0$ or {\em is not aperiodic} if there exists $m \not= n \in \N^k$ such that for all $x \in Z(v), \ \sigma^m(x) = \sigma^n(x)$.  We define
\[\text{Per}(v) := \Z \{ m-n \mid \sigma^m(x) = \sigma^n(x), \ \forall \ x\in Z(v)\}.\]

If $\Lambda$ is strongly connected, then \cite[Lemma 5.1]{aHLRS} establishes that
\[\text{Per}(v) = \text{Per}(w) = : \text{Per}\Lambda\]
 for any $v, w \in \Lambda^0$.  In fact,
\begin{equation}
\label{eq:Per-Lambda}
\text{Per}\Lambda = \{ d(\lambda) - d(\nu) \mid \lambda x = \nu x, \ \forall \ x \in \Lambda^\infty\}.
\end{equation}
We define $P_\Lambda := \{ (\lambda, \nu) \in \Lambda \times \Lambda \mid \lambda x = \nu x, \ \forall \ x \in \Lambda^\infty\}$.
\end{defn}

\begin{rmk}
\label{rmk:aHLRS-Prop-8-2}
We observe that Lemma 8.4 of \cite{aHLRS}, and its proof, are still valid if we replace the measure $M$ by $\mu_{y,\theta}$, and replace $\rho(\Lambda)^{-d(\lambda)}$ by $e^{-\theta y(\lambda)} \rho(B(y,\theta))^{-d(\lambda)}$ wherever the former appears in the proof.  The key idea of this proof is that $\{ M(Z(v)): v \in \Lambda^0\}$ is an eigenvector for each matrix $A_i$, with eigenvalue $\rho(\Lambda)^{e_i}$.  In our case, Proposition \ref{prop:unique-measure} guarantees that $\{ \mu_{y,\theta}(Z(v)): v\in \Lambda^0\}$ is an eigenvector for each matrix $B_i(y,\theta)$ with eigenvalue $ \rho(B_i(y,\theta))$.
Consequently, Proposition 8.2 of \cite{aHLRS} also holds for the measures $\mu_{y,\theta}$, so we conclude that
\begin{equation}
\label{eq:periodic-paths-dense}
\mu_{y,\theta}\left( \{ x \in \Lambda^\infty \mid \sigma^m(x) = \sigma^n(x)\}\right) = \begin{cases}
1, & m-n \in \text{Per} \Lambda \\
0, & \text{ else.}
\end{cases}
\end{equation}
\end{rmk}

\section{KMS states associated to $(y,\theta)$}
\label{sec:kms}

Throughout this section $\Lambda$ will be a finite, strongly connected $k$-graph.

Each $\R_+$-functor $y$ gives us a family of generalized gauge actions $\{\alpha^{y,\theta}\}_{\theta > 0}$ on $C^*(\Lambda)$ (denoted $\overline{\alpha}^{\overline y}$ in Remark 5.25 of \cite{McNamara}):  For $t \in \R$ and any generator $s_\lambda$ of $C^*(\Lambda)$,
\begin{equation}\alpha^{y,\theta}_t(s_\lambda) := e^{i t y(\lambda)}\left( \rho(B(y,\theta))^{d(\lambda)} \right)^{it/\theta} s_\lambda = e^{it(y(\lambda) + \frac{1}{\theta} \ln(\rho(B(y,\theta))^{d(\lambda)}))} s_\lambda.
\label{eq:action}
\end{equation}
In general, these actions are not the same as the actions $\alpha^r$  considered in \cite{aHLRS-coord-irred} and \cite{aHLRS}. 
 As was remarked already in \cite{McNamara},  $\alpha^{y,\theta}$ differs from $\alpha^{r}$ as soon as $y$ is not of the form $y(\lambda)=r\cdot d(\lambda)$ for some fixed $r\in(0,\infty)^k$ and all $\lambda\in \Lambda$.  Section \ref{sec:examples} describes a variety of examples of $\R_+$-functors $y$ which are not of the form $y(\lambda) = r \cdot d(\lambda)$.

In this section, we will compute the KMS$_\beta$ states associated to the actions $\alpha^{y,\theta}$.
Proposition \ref{prop:KMS-s-states} establishes that  KMS$_{\beta} $ states for $\alpha^{y,\theta}$ exist precisely when $\alpha^{y,\theta} = \alpha^{y,\beta }$.  The KMS$_\beta$ states for $\alpha^{y,\theta}$ are therefore described in Theorem \ref{thm:KMS-simplex} below. A similar reduction in computational complexity occurs for the actions $\alpha^r$ mentioned above.  Indeed, as explained  at the beginning of Section 7 of \cite{aHLRS},   any KMS$_\beta$ state associated to an action of the form $\alpha^r$ must also be a KMS$_1$ state for $\alpha^{\ln(\rho(\Lambda))}$.

In order to compute the KMS states for a $C^*$-dynamical system $(A, \gamma)$, one often first identifies a good dense subset of $A$ on which $\gamma$ is analytic (it is well-known that the $\gamma$-analytic elements are dense in $A$, see for example \cite{bra-robII}, but for Cuntz-Krieger type algebras there usually is a good spanning set for $A$ inside the analytic elements).  In our setting, this dense subset is $\text{span} \{ s_\lambda s_\nu^*: \lambda, \nu \in \Lambda\}$.
Indeed, as observed in \cite[page 269]{aHLRS-coord-irred} for the case of  the gauge action, any element of $C^*(\Lambda)$ of the form $s_\lambda s_\nu^*$ is  $\alpha^{y,\theta}$-analytic for every $\R_+$-functor $y$ and all $\theta> 0$.  To see this, note that the function $t\mapsto \alpha_t^{y,\theta}(s_\lambda s_\nu^*)$ on $\R$ admits the extension
\[
\zeta \mapsto e^{i\zeta\left(y(\lambda)-y(\nu))+\ln (\rho(B(y,\theta))^{(d(\lambda)-d(\nu))/\theta}\right)} s_\lambda s_\nu^*,\]
which is an entire function.  The fact that $\text{span} \{ s_\lambda s_\nu^*: \lambda, \nu \in \Lambda\}$ is dense in $C^*(\Lambda)$ thus implies that the KMS$_{\beta}$ states for $\alpha^{y,\theta}$ are precisely the norm-one positive linear functionals $\phi: C^*(\Lambda) \to \C$ such that
for any $(\lambda, \nu), (\rho, \eta) $ with $s(\lambda) = s(\nu)$ and $s(\rho) = s(\eta)$,
\begin{equation}\label{eq:KMS condition on spanning elements} \phi(s_\lambda s_\nu^* s_\rho s_\eta^*) = \phi(s_\rho s_\eta^* \alpha_{i\beta}^{y,\theta}( s_\lambda s_\nu^*)).
\end{equation}

{We next observe that in the groupoid picture of $k$-graph algebras, these actions $\alpha^{y,\theta}$ arise from a cocycle $c_{y,\theta}$ on the $k$-graph groupoid $\G_\Lambda$ as in \cite{neshveyev} or \cite{thomsen};
\begin{equation} c_{y,\theta}(\lambda z, d(\lambda) - d(\nu), \nu z) = y(\lambda) - y(\nu) + \log\left( \rho\left( B(y, \theta) \right)^{(d(\lambda) - d(\nu))/\theta}\right).
\label{eq:cocycle-y-beta}
\end{equation}

\begin{prop}
The function $c_{y, \theta}: \G_\Lambda \to \R$ is well-defined and satisfies $c_{y,\theta}(gh) = c_{y,\theta}(g) + c_{y,\theta}(h).$
\label{prop:cocycle-well-def}
\end{prop}
\begin{proof}
Suppose that $g = (\lambda z, d(\lambda) - d(\nu), \nu z) = (\lambda' z', d(\lambda') - d(\nu'), \nu'z')$.
To show that $c_{y,\theta}$ is well defined, it suffices to show that $y(\lambda) - y(\nu) = y(\lambda') - y(\nu')$.  To that end,
write $d(\lambda) - d(\lambda') = (n_1, \ldots, n_k) \in \Z^k$ and, for each $1 \leq i \leq k$, define
\[(m_\lambda)_i = \max\{ 0, -n_i\}; \qquad (m_{\lambda'})_i = \max\{ 0, n_i\}.\]
Then [defining $m_\lambda \in \N^k$ by $m_\lambda = ((m_\lambda)_i)_{i=1}^k$] we have $d(\lambda) + m_\lambda = d(\lambda') + m_{\lambda'}$ and, since $d(\nu) - d(\nu') = d(\lambda) - d(\lambda') = (n_1, \ldots, n_k)$, we also have  $d(\nu) + m_\lambda = d(\nu') + m_{\lambda'}.$

Write $\lambda_z = z(0, m_\lambda)$ and $\lambda'_{z'} = z'(0, m_{\lambda'})$.  By construction,
\[ \lambda \lambda_z = \lambda' \lambda'_{z'} \quad \text{ and } \quad  \nu \lambda_z = \nu' \lambda'_{z'}.\]
The additivity of $y$ now implies that, as desired,
\[ y(\lambda) - y(\nu) = y(\lambda \lambda_z) - y(\nu \lambda_z) = y(\lambda' \lambda'_{z'}) - y(\nu' \lambda'_{z'}) = y(\lambda') - y(\nu'). \]

Showing that $c_{y,\theta}$ is multiplicative uses a similar argument.  Given $g, h \in \G_\Lambda$ with $s(g) = r(h)$, we can choose  $ \lambda_g, \nu_g = \lambda_h, \nu_h \in \Lambda$ such that
\[ g = (\lambda_gz, d(\lambda_g) - d(\nu_g), \nu_g z), \quad h = (\lambda_h z, d(\lambda_h) - d(\nu_h), \nu_h z),\]
\[  gh = (\lambda_{g} z, d(\lambda_{g}) - d(\nu_{h}), \nu_{h}z);\]
cf.~\cite[Lemma 6.3]{kps-twisted}.  To check that $c_{y,\theta}$ is multiplicative, 
it now suffices to observe that
\[y(\lambda_g) - y(\nu_g) + y(\lambda_h) - y(\nu_h) = y(\lambda_g) - y(\nu_h). \qedhere\]
\end{proof}

The following definition is an application of the definition of quasi-invariance from \cite{neshveyev} to our setting, taking for our sets $U$ the basic open sets $Z(\lambda,\nu) = \{ (\lambda z, d(\lambda) - d(\nu), \nu z): z \in \Lambda^\infty\}$ of $\G_\Lambda$.  We have also invoked the fact that $c_{y,\theta}$ is constant on the sets $Z(\lambda, \nu)$.

\begin{defn}
\label{def:quasi-invariant}
Let $y$ be a $\R_+$-functor and $\theta,\beta \in (0,\infty)$. We say that
a measure $\mu$ on $\Lambda^\infty$ is {\em quasi-invariant} with Radon-Nikodym cocycle $e^{-\beta c_{y,\theta}}$ if, for any $(\lambda, \nu) \in \Lambda\times \Lambda$ with $s(\lambda) = s(\nu)$ and all $z \in \Lambda^\infty$,
\[ e^{-\beta c_{y,\theta}(\lambda z, d(\lambda) - d(\nu), \nu z)} \mu (Z(\nu)) = \mu(Z(\lambda)).\]
Equivalently, $\mu$ is quasi-invariant with Radon-Nikodym cocycle $e^{-\beta c_{y,\theta}}$ iff
\[ e^{\beta y(\nu)} \rho(B(y,\theta))^{\frac{\beta} {\theta}d(\nu)} \mu(Z(\nu)) =  e^{\beta y(\lambda)} \rho(B(y,\theta))^{\frac{\beta} {\theta}d(\lambda)} \mu(Z(\lambda))\]
whenever $s(\nu) = s(\lambda)$.
\end{defn}
\begin{rmk}
\label{rmk:quasi-invar}
\begin{enumerate}
\item If $\beta  = \theta$, the measure $\mu_{y,\beta}$ is quasi-invariant with Radon-Nikodym cocycle $e^{-\theta c_{y,\theta}}$.
In fact, Proposition \ref{prop:unique-measure} establishes that $\mu_{y,\theta}$ is the unique such measure.
\item
It is relatively straightforward to check that $\mu_{y,\theta}$ is ($\G_\Lambda, c_{y,\theta})$ conformal in the sense of \cite{thomsen}.
 Consequently, Proposition \ref{prop:integration-KMS} below could also be derived from  \cite[Proposition 2.1]{thomsen}.
\end{enumerate}
\end{rmk}

\begin{prop}
\label{prop:integration-KMS}
Let $\Lambda$ be a finite, strongly connected $k$-graph, equipped with a $\R_+$-functor $y$, and choose $\beta , \theta \in (0, \infty)$.
Let $\Psi$ denote the canonical conditional expectation $\Psi: C^*(\Lambda) \to C(\Lambda^\infty)$.
 For any quasi-invariant probability measure $\mu$ with Radon-Nikodym cocycle $e^{-\beta  c_{y,\theta}}$, the function
 \[ \psi_{\beta} (a) := \int_{\Lambda^\infty} \Psi(a) \, d\mu, \]
for $a \in C^*(\Lambda)$, defines a KMS$_{\beta} $ state for $(C^*(\Lambda), \alpha^{y,\theta})$.
 \end{prop}
 \begin{proof}
Since $\mu$ is a probability measure, $\psi_{\beta} $ is easily verified to be a positive linear functional of norm 1.
Since the elements $s_\lambda s_\nu^*$ in a dense spanning family for $C^*(\Lambda)$ are $\alpha^{y,\theta}$-analytic, it suffices to verify condition \eqref{eq:KMS condition on spanning elements}.

One easily sees that
\[\alpha^{y,\theta}_{i\beta }(s_\lambda s_\nu^*) = e^{-\beta (y(\lambda) - y(\nu))} \rho(B(y, \theta))^{\frac{\beta} {\theta} (d(\nu) - d(\lambda))} s_\lambda s_\nu^*,\]
  so, using Equation \eqref{eq:s-alpha s-beta},
$\psi_{\beta} (s_\gamma s_\zeta^* \alpha^{y,\theta}_{i\beta }(s_\lambda s_\nu^*))$ becomes
\begin{align*}  { }& e^{-\beta (y(\lambda) - y(\nu))} \rho(B(y, \theta))^{\frac{\beta} {\theta}d(\nu) - d(\lambda)}  \sum_{(\rho_1, \rho_2) \in \Lambda^{min}(\zeta, \lambda)} \psi_{\beta} (s_{\gamma \rho_1} s_{\nu \rho_2}^*) \\
&= e^{-\beta (y(\lambda) - y(\nu))} \rho(B(y, \theta))^{\frac{\beta} {\theta}d(\nu) - d(\lambda)}  \sum_{(\rho_1, \rho_2) \in S_1} \mu(Z(\nu \rho_2)),
\end{align*}
where $S_1 = \{(\rho_1, \rho_2) \in \Lambda^{min}(\zeta, \lambda): \gamma \rho_1 = \nu \rho_2\}$.

Similarly, setting $S_2 = \{(\eta_1, \eta_2) \in \Lambda^{min}(\gamma, \nu): \lambda \eta_2 = \zeta \eta_1\}$,
\begin{align*}
\psi_{\beta} (s_\lambda s_\nu^* s_\gamma s_\zeta^*)&= \sum_{(\eta_1, \eta_2) \in \Lambda^{min}(\gamma, \nu)} \psi_{\beta} (s_{\lambda \eta_2} s_{\zeta \eta_1}^*) \\
&= \sum_{(\eta_1, \eta_2) \in S_2} \mu(Z(\lambda \eta_2)).
\end{align*}

Note that, since $(\eta_1, \eta_2) $ is an extension of $(\zeta, \lambda)$ for any $(\eta_1, \eta_2) \in S_2$, we must have
\[ d(\eta_i) \geq d(\rho_i) \text{ for  } i = 1, 2\]
for any  $(\rho_1, \rho_2) \in \Lambda^{min}(\zeta, \lambda)$.
Similarly, since $(\rho_1, \rho_2)$ is an extension of $(\gamma, \nu)$ for any $(\rho_1, \rho_2) \in S_1$, we must have $d(\rho_i) \geq d(\eta_i)$ for each $i$.  Hence $d(\rho_i) = d(\eta_i)$ for all $(\rho_1, \rho_2) \in S_1, (\eta_1, \eta_2) \in S_2$.  In other words, $(\rho_1, \rho_2)$ is a minimal extension of $(\gamma, \nu)$ and $(\eta_1, \eta_2)$ is a minimal extension of $(\zeta, \lambda)$, for all $(\rho_1, \rho_2) \in S_1, (\eta_1, \eta_2) \in S_2$.  It follows that
\[ S_1 = \Lambda^{min}(\zeta, \lambda) \cap \Lambda^{min}(\gamma, \nu)  =S_2 .\]

Now, we  see that $\psi_{\beta} (s_\gamma s_\zeta^* \alpha^{y,\theta}_{i\beta }(s_\lambda s_\nu^*))$ transforms as
\begin{align*}
 { }& e^{-\beta (y(\lambda) - y(\nu))} \rho(B(y, \theta))^{\frac{\beta} {\theta}(d(\nu) - d(\lambda))}  \sum_{(\rho_1, \rho_2) \in S_1} e^{-\beta  y(\nu)} \rho(B(y,\theta))^{-\frac{\beta} {\theta}d(\nu)}  \mu(Z(\rho_2)) \\
&= e^{-\beta  y(\lambda)} \rho(B(y,\theta))^{-\frac{\beta} {\theta}d(\lambda)}  \sum_{(\eta_1, \eta_2) \in S_2 = S_1} \mu (Z(\eta_2))\\
&= \sum_{(\eta_1, \eta_2) \in S_2} \mu(Z(\lambda \eta_2) )= \psi_{\beta} (s_\lambda s_\nu^* s_\gamma s_\zeta^*),
\end{align*}
which is \eqref{eq:KMS condition on spanning elements}. Hence $\psi_{\beta} $ is a KMS$_{\beta} $ state as claimed.
 \end{proof}

\begin{prop}
\label{prop:KMS-s-states}
Fix a finite, strongly connected $k$-graph $\Lambda$, a $\R_+$-functor $y$ on $\Lambda$, and $\beta , \theta \in (0,\infty)$. There exist  KMS$_{\beta }$ states for $(C^*(\Lambda), \alpha^{y,\theta})$ iff $\alpha^{y,\theta} = \alpha^{y,\beta }$.
\end{prop}
\begin{proof}
Choose  a KMS$_{\beta }$ state $\phi$ for $(C^*(\Lambda), \alpha^{y,\theta})$.
By the Cuntz-Krieger relations and the KMS condition, for any $1 \leq i \leq k$ and $v \in \Lambda^0$,
\begin{align*}
\phi(p_v) &= \sum_{\lambda \in v\Lambda^{e_i}} \phi(s_\lambda s_\lambda^*) = \sum_\lambda \phi(s_\lambda^* s_\lambda) e^{-\beta  y(\lambda)} \left( \rho(B(y,\theta))^{-d(\lambda)}\right)^{\beta /\theta} \\
&= \sum_{\lambda \in v\Lambda^{e_i}} \phi(p_{s(\lambda)})  e^{-\beta y(\lambda)} \rho(B_i(y,\theta))^{-\beta /\theta}\\
&= \rho(B_i(y,\theta))^{-\beta /\theta}\sum_{w\in \Lambda^0} \phi(p_w) \sum_{\lambda \in v\Lambda^{e_i} w} e^{-\beta y(\lambda)}.
\end{align*}
Since $\sum_{\lambda \in v\Lambda^{e_i}w} e^{-\beta  y(\lambda)} = B_i(y,\beta )_{v,w}$, we see that $(\phi(p_v))_{v\in \Lambda^0}$ is an eigenvector for  $B_i(y, \beta )$ with eigenvalue $ \rho(B_i(y,\theta))^{\beta /\theta}.$  Moreover,
\[ \sum_{v \in \Lambda^0} \phi(p_v) = \phi(1) = 1,\]
so \cite[Proposition 3.1(a)]{aHLRS} implies that $(\phi(p_v))_{v\in \Lambda^0} = \xi^{y,\beta } $ and
\[
 \rho(B_i(y,\theta))^{1/\theta} = \rho(B_i(y,\beta ))^{1/\beta } \ \  \forall \ i. \]
It follows that $\alpha^{y,\theta} = \alpha^{y,\beta }$.

Conversely, if $\alpha^{y,\beta } = \alpha^{y,\theta}$, then the KMS$_{\beta }$ states for the two actions are the same (and constitute a nontrivial set by Proposition \ref{prop:integration-KMS}).
\end{proof}
In the case when $\Lambda$ has only one vertex, we obtain a slightly sharper result.  Such higher-rank graphs have been extensively studied by Davidson, Power, and Yang (cf.~\cite{davidson-yang, davidson-yang-repns, davidson-power-yang-dilation}).

\begin{prop}
\label{prop:one-vertex-unique-beta}
Let $\Lambda$ be a finite  $k$-graph with one vertex.  Choose a $\R_+$-functor $y$ on $\Lambda$ and $\beta , \theta \in (0,\infty)$. If there exist KMS$_\beta$ states for $\alpha^{y,\theta}$, then $\beta = \theta$.
\end{prop}

\begin{proof}
In the one-vertex case, each adjacency matrix $B_i(y,\theta)$  has only one (positive) entry, which is also its spectral radius:
		\[
		\rho(B_i(y,\theta))= \sum_{h\in \Lambda^{e_i}} e^{-\theta y(h)}.
		\]
Since the function $
		\theta \to \rho(B_i(y,\theta))$
		is differentiable by Theorem \ref{thm:spec-rad-cts}, the function
		\[
		\psi_i(\theta):=   \rho(B_i(y,\theta))^{1/\theta}= \Big(\sum_{h \in \Lambda^{e_i}} e^{-\theta y(h)}\Big)^{1/\theta}
		\]
		is also differentiable, and
\begin{align*}
\frac{d\psi_i}{d\theta}  & = \frac{1}{\theta} \left( \rho(B_i(y, \theta)) \right)^{\frac{1-\theta}{\theta}} \frac{d}{d\theta} \rho(B_i(y,\theta))= \frac{1}{\theta} \left( \rho(B_i(y, \theta)) \right)^{\frac{1-\theta}{\theta}}
		\sum_{h \in \Lambda^{e_i}} -y(h) e^{-\theta y(h)}\\
		& < 0.
\end{align*}
Consulting the formula \eqref{eq:action} for the action $\alpha^{y,\theta}$ reveals then that $\alpha^{y,\beta} \not= \alpha^{y,\theta}$ if $\beta \not= \theta \in \R_+$.  The result now follows from Proposition \ref{prop:KMS-s-states}.
\end{proof}

 It would be interesting to know if the conclusion of Proposition \ref{prop:one-vertex-unique-beta} is valid for $k$-graphs with more than one vertex.  Proposition \ref{prop-partial-result-equality-thetas} offers a partial result in this direction.  In his thesis, McNamara  identified a different set of hypotheses guaranteeing the uniqueness result of Proposition \ref{prop:one-vertex-unique-beta}; we discuss this result in   Remark~\ref{rmk:McNamaras uniqueness} below.

\begin{prop}
\label{prop-partial-result-equality-thetas}
Fix a finite, strongly connected $k$-graph $\Lambda$, a $\R_+$-functor $y$ on $\Lambda$,
and suppose that  for some interval $(a,b) \subset \R_+$
\[
\rho(B_i(y,\theta) ) >1, \ \forall\,  \theta \in (a,b), \ \forall \, 1 \leq i \leq k.
\]
For $\beta , \theta \in (a,b)$, there exist  KMS$_{\beta }$ states for $(C^*(\Lambda), \alpha^{y,\theta})$ iff $\beta = \theta$.
\end{prop}

\begin{proof}
This follows from the explicit  computation of the derivative of
\[	\psi_i(\theta)=   \rho(B_i(y,\theta))^{1/\theta};\]
we will show that $\frac{d\psi_i}{d\theta} < 0$ on the entire interval $(a, b)$  under the given hypotheses.
Once this is established,  the same argument used in the proof of Proposition \ref{prop:one-vertex-unique-beta} ends the proof.

To show $\frac{d\psi_i}{d\theta} <0$ on $(a,b)$,  we explicitly compute
\begin{align*}
\frac{d\psi_i}{d\theta}  & = \frac{d}{d\theta} \left( e^\frac{\ln \rho(B_i(y,\theta))}{\theta}   \right) \\
& = \rho(B_i(y,\theta))^{1/\theta} \left[ -\frac{1}{\theta^2}  \ln( \rho(B_i(y,\theta))) + \frac{\frac{d}{d\theta} \left(\ln \rho(B_i(y,\theta))\right)}{\theta} \right] .
\end{align*}
Now  $ \rho(B_i(y,\theta) ) >1
$  implies  $\ln( \rho(B_i(y,\theta)))>0 .$ Therefore
\begin{align*}
\frac{d\psi_i}{d\theta}  <0\  &\iff\  \ \left[ -\frac{1}{\theta}  \ln( \rho(B_i(y,\theta))) + {\frac{d}{d\theta}\left(\ln \rho(B_i(y,\theta))\right)}{} \right] <0\\
 &\iff\  \ \frac{\frac{d}{d\theta}\left(\ln \rho(B_i(y,\theta))\right)}{ \ln( \rho(B_i(y,\theta)))     }  <\, \frac1\theta.\\
\end{align*}
Gelfand's formula implies that  $\rho(B_i(y,\theta)) $ is a non-increasing function of $\theta$.  Taking derivatives reveals that $\ln (\rho(B_i(y,\theta))) $ is also a non-increasing function of $\theta$. It follows that
\[ \frac{\left(\ln \rho(B_i(y,\theta))\right)'}{ \ln( \rho(B_i(y,\theta)))} \, \leq \, 0 < \frac{1}{\theta},\]
thus proving $\frac{d\psi_i}{d\theta} <0.$
\end{proof}

For the remainder of the section, we assume a $\R_+$-functor $y$ and $\beta \in (0,\infty)$ are given.  We  will characterize the KMS$_\beta$ states of  $(C^*(\Lambda), \alpha^{y,\beta})$. Recall from \cite{bra-robII} that the KMS$_\beta$-states for $\beta\in \mathbb{R}$ form a Choquet simplex, and in particular are determined by the extremal KMS$_\beta$-states.

In the terminology of \cite{christensen} or \cite{neshveyev},  Proposition \ref{prop:unique-measure} above implies that $\mu_{y,\beta}$ is the unique measure on $\Lambda^\infty$ which is $e^{-\beta c_{y,\beta}}$-quasi-invariant.
Therefore, Lemma 4.1 and Theorem 5.2 of \cite{christensen} imply that the extremal KMS$_\beta$ states for $\alpha^{y,\beta}$ are in bijection with  a certain subgroup $\widehat B$ of $\T^k$.  Theorem \ref{thm:KMS-simplex} below establishes that $B$ is equal to $\text{Per}\Lambda \subseteq \Z^k$, the periodicity group of $\Lambda$.

We observe that Christensen's Theorem 5.2 is a refinement of Neshveyev's description \cite{neshveyev} of KMS states on groupoid $C^*$-algebras.  While both Christensen and Neshveyev use quasi-invariant measures to parametrize KMS states, Christensen replaces Neshveyev's measurable fields of states with the group $B$, which consists of  symmetries of the simplex of KMS$_\beta$ states.

In the case that $y\equiv 0$ and $\beta=1$, so that $\alpha^{y,\beta}$ is the preferred dynamics on $C^*(\Lambda)$, Theorem \ref{thm:KMS-simplex} reduces to \cite[Theorem 7.1]{aHLRS}. Our proof of Theorem \ref{thm:KMS-simplex} combines \cite[Theorem 7.1]{aHLRS} with \cite[Theorem 5.2]{christensen} to show that $C^*(\text{Per}\,\Lambda)$ parametrizes KMS$_\beta$ states for any of the dynamics in the family $\{\alpha^{y,\beta}\}_{\beta> 0}$.

\begin{thm}
\label{thm:KMS-simplex}
Let $\Lambda$ be a finite, strongly connected $k$-graph, with a $\R_+$-functor $y$, and fix $\beta \in (0, \infty)$.  The simplex of KMS$_\beta$ states for $(C^*(\Lambda), \alpha^{y,\beta})$ is affinely isomorphic to the state space of $C^*(\text{Per}\,\Lambda)$.
\end{thm}
\begin{proof}
Recall the isomorphism $C^*(\G_\Lambda) \cong C^*(\Lambda)$ sending $\chi_{Z(\rho, \eta)}$ from the dense subalgebra $C_c(\G_\Lambda) \subseteq C^*(\G_\Lambda)$ to $s_\rho s_\eta^* \in C^*(\Lambda)$, where
\[Z(\rho, \eta) := \{ (x, d(\rho) -d(\eta), y) \mid x \in Z(\rho), y \in Z(\eta)\} \subseteq \G_\Lambda
\]
for $\rho, \eta \in \Lambda$. Recall also that $\alpha^{y,\beta}$ is determined by the cocycle $c_{y,\beta}$. We aim to apply \cite[Theorem 5.2]{christensen} to $\mathcal{G}_\Lambda$ and $c_{y,\beta}$. For this we must first  identify the extremal KMS$_\beta$ states of $(C^*(\mathcal{G}_\Lambda), \alpha^{y,\beta})$; these are given as in Equation (5.1) of \cite[Theorem 5.1]{christensen} by integrating functions $f\in C_c(\G_\Lambda)$ with respect to measures which are extremal points in the set of  $e^{-\beta c_{y,\beta}}$-quasi-invariant measures on $\Lambda^\infty$. In our case, there is a unique such measure, namely $\mu_{y,\beta}$, see Proposition \ref{prop:unique-measure}.  Since any KMS state must be linear and continuous, we may therefore assume that Christensen's functions $f$ are of the form $\chi_{Z(\rho,\eta)}$ for $\rho,\eta \in \Lambda$. Equivalently, we apply  Equation (5.1) of \cite{christensen} to monomials $s_\rho s_\eta^*$ in $C^*(\Lambda)$.
 Note also that in our setting, Christensen's group $A$ is $ \Z^k$ and $\Phi: \G_\Lambda \to A$ is given by $\Phi(x, n,y) = n$.

Assume first that the monomial $s_\lambda s_\nu^*$ corresponds to a periodic pair $(\lambda, \nu) \in P_\Lambda$ as in Definition \ref{def:periodicity}. Then  $Z(\lambda) = Z(\nu)$, and  Equation (5.1) of \cite{christensen} implies that any extremal KMS$_\beta$ state $\omega$ for $\alpha^{y, \beta}$ must satisfy 
\[ \omega ( s_\lambda s_\nu^*)  = z^{d(\lambda) - d(\nu)} \mu_{y,\beta}(Z(\lambda)) = z^{d(\lambda) - d(\nu)}\mu_{y,\beta}(Z(\nu))\]
for some $z \in \T^k$.  Write $\omega_1$ for the extremal KMS$_\beta$ state associated to $z = (1, 1, \ldots, 1)$, cf.~\cite[Proposition 4.2]{christensen}.

For an arbitrary pair $(\rho,\eta)$, Theorem 5.1 of \cite{christensen} reveals that every extremal KMS$_\beta$ state for $\alpha^{y,\beta}$ will be of the form
\[ \omega (s_\rho s_\eta^*) = z^{d(\rho) - d(\eta)} \mu_{y,\beta}(A_{\rho,\eta}),\]
where $A_{\rho,\eta}=\{ x\in Z(\rho) \cap Z(\eta) \mid \sigma^{d(\rho) +m}(x) = \sigma^{d(\eta)+m}(x) \text{ for some } m \in \N^k\}$.
However, by Remark \ref{rmk:aHLRS-Prop-8-2} we see that $\mu_{y,\beta}(A_{\rho,\eta})$ will be zero unless $d(\rho) - d(\eta) \in \text{Per}\Lambda$.

We claim that if $d(\rho) - d(\eta) \in \text{Per}\Lambda$ and $Z(\rho) \cap Z(\eta) \not= \emptyset$, then $(\rho, \eta) \in P_\Lambda$.  Indeed, with $v=r(\rho)=r(\eta)$, the fact that $\sigma^{d(\rho)}(y)=\sigma^{d(\eta)}(y)$ for all $y\in Z(v)$ implies,  by Lemma 5.1(b) of \cite{aHLRS}, that there is a unique $\tilde{\eta} \in \Lambda^{d(\eta)}$ such that $\rho x'=\tilde{\eta} x'$ for all $x'\in Z(s(\rho))$. This in particular means that $(\rho, \tilde{\eta}) \in P_\Lambda$.  Therefore, if there exists $x = \rho  x_1 = \eta x_2 \in Z(\rho) \cap Z(\eta)$ for some $x_1, x_2 \in \Lambda^\infty$, then
\[ \tilde{\eta} x_1 = \rho x_1 = \eta x_2 \Rightarrow \eta x_2(0, m) = \tilde{\eta} x_1(0,m)\]
for any $m \in \N^k$. 
Since $d(\eta) = d(\tilde{\eta})$, it follows that $\tilde{\eta} = \eta$ and $x_1 = x_2$. 

  In other words, any extremal KMS$_\beta$ state for $\alpha^{y, \beta}$ must be of the form
\begin{equation}
\omega_z(s_\rho s_\eta^*) = \begin{cases}
z^{d(\rho) - d(\eta)} \mu_{y,\beta}(Z(\rho)) = z^{d(\rho) - d(\eta)} \mu_{y,\beta}(Z(\eta)), & (\rho,\eta) \in P_\Lambda\\
0, & \text{else}
\end{cases}
\label{eq:extremal-kms-formula}
\end{equation}
for some $z \in \T^k$.

Theorem 5.2(2) of \cite{christensen} describes a homeomorphism between the dual $\widehat B$ of a certain subgroup $B \subseteq A = \Z^k$ and the set of extremal KMS$_\beta$ states for $\alpha^{y,\beta}$
 More precisely, $B$ is defined in terms of a subgroup $N \subseteq \T^k$.  In our setting,  
\[ N = \{ z \in \T^k: \omega_1 (s_\rho s_\eta^*) = z^{d(\rho) -d(\eta)} \omega_1(s_\rho s_\eta^*) \ \forall \ (\rho, \eta) \in \Lambda *_s \Lambda\}.\]
Since $\omega_1(s_\rho s_\eta^*) = 0$ unless $(\rho, \eta) \in P_\Lambda$, it follows that
\begin{equation}
\label{eq:N}
N= \{ z \in \T^k: z^{d(\rho) - d(\eta)} = 1 \ \forall \ (\rho, \eta) \in P_\Lambda\}.
\end{equation}
Consequently,
\[ B = N^\perp =  \{ m \in \Z^k: z^m =1 \ \forall \ z \in N\}\]
clearly contains $\text{Per}\Lambda$.

Our goal is to show that $B =\text{Per}\Lambda$.  To this end, note that $N$ is independent of the choice of $y$ and $\beta$.  Furthermore, when $y=0$ we obtain $B_i(0, \beta) =A_i$, for any $\beta >0$.  
  Thus,   $\mu_{0,\beta}$ agrees with the measure $M$ from Proposition 8.1 of \cite{aHLRS}, and $\alpha^{0,1}$ agrees with the preferred dynamics $\alpha$ from \cite{aHLRS}.

   Theorem 7.1 of \cite{aHLRS} establishes that the extremal KMS$_1$ states for $\alpha$ are in bijection with the pure states of $C_0( \widehat{\text{Per}\, \Lambda})$, that is, with the points of $\widehat{\text{Per}\, \Lambda}$.  Moreover, Remark 10.4 of \cite{aHLRS} shows that this bijection (just as in Theorem 5.2 of \cite{christensen}) assigns the state
\[ \phi_z(s_\lambda s_\nu^*) = \begin{cases}
M(Z(\lambda)) z^{d(\lambda) - d(\nu)}, & (\lambda, \nu) \in P_\Lambda \\
0, & \text{ else}
\end{cases}\]
to $z \in  \widehat{\text{Per}\, \Lambda}.$
Therefore, we must have $\widehat{B}  = \widehat{\text{Per}\, \Lambda}$ and hence $B = \text{Per}\, \Lambda$.  Consequently, for any  $\R_+$-functor $y$ and $\beta \in \R_+$, the simplex of KMS$_\beta$ states of $(C^*(\Lambda), \alpha^{y,\beta})$ is affinely isomorphic to the state space of $C^*(\text{Per}\, \Lambda)$.
\end{proof}
\begin{cor}
\label{cor:simplicity}
Let $\Lambda$ be a finite, strongly connected higher-rank graph, and fix $\beta \in (0, \infty)$ and an $\R_+$-functor $y$ on $\Lambda$.  The $C^*$-dynamical system $(C^*(\Lambda), \alpha^{y,\beta})$ admits a unique KMS$_\beta$ state iff $\Lambda$ is aperiodic, iff $C^*(\Lambda)$ is simple.
\end{cor}
\begin{proof}
The last equivalence was established in  \cite[Theorem 11.1]{aHLRS}.  For the first equivalence, observe that by Theorem \ref{thm:KMS-simplex},
uniqueness of the KMS$_\beta$ state is equivalent to the triviality of $C^*(\text{Per}\,\Lambda)$ -- in other words, to the aperiodicity of $\Lambda$.
\end{proof}

\begin{rmk}\label{rmk:McNamaras uniqueness} In his thesis \cite{McNamara}, McNamara considers finite $k$-graphs $\Lambda$ which are {\em coordinatewise irreducible}, in the sense that each coordinate matrix   $A_i$, for $i=1,\dots, k$, is irreducible.   In particular,  \cite[Theorem 5.30]{McNamara}  establishes that given a $\R_+$-functor $y$ on such a $k$-graph $\Lambda$ and $\beta\in(0,\infty)$, if the statement
\begin{equation}
\label{eq:needed-for-aperiodicity}
y(\lambda) + \frac{1}{\beta}  \ln (\rho(B(y,\beta))^{d(\lambda)} )= y(\nu) + \frac{1}{\beta}  \ln (\rho(B(y,\beta))^{d(\nu)} )\Rightarrow d(\lambda) = d(\nu)
\end{equation}
holds, then
 there is a unique KMS state $\phi$ of $(C^*(\Lambda), \alpha^{y,\beta})$ occurring at (inverse) temperature $\beta$.  Moreover, this KMS state satisfies
$\phi(s_\lambda s_\nu^*)=\delta_{\lambda,\nu}e^{-\beta y(\lambda)} \left( \rho(B(y,\beta))^{-d(\lambda)}\right)^{1/\beta}\xi_{s(\lambda)}^{y,\beta}$.
Observe  that $\phi$ is the state we obtained in Proposition~\ref{prop:integration-KMS} above.
\end{rmk}
Combining McNamara's result with our Theorem~\ref{thm:KMS-simplex} above implies that if $\Lambda$ is coordinatewise irreducible and Equation \eqref{eq:needed-for-aperiodicity} holds for at least one pair $(y,\beta)$, then $\Lambda$ must be aperiodic. Given the potential importance of this result for applications, we also offer a direct proof which does not rely on  Theorem \ref{thm:KMS-simplex}.

\begin{cor}
\label{cor:aperiodicity}
Let $\Lambda$ be a finite, coordinatewise irreducible $k$-graph.   If there exists an $\R_+$-functor $y$ on $\Lambda$ and $\beta \in (0, \infty)$ such that Equation \eqref{eq:needed-for-aperiodicity} holds,
then $\Lambda$ is aperiodic.
\end{cor}
\begin{proof}
We argue by contrapositive.  If $\text{Per}\,\Lambda \not= 0$, choose $(\lambda, \nu) \in P_\Lambda$ with $d(\lambda) \not= d(\nu)$.  By construction, $Z(\lambda) = Z(\nu)$, so
\[ \mu_{y,\beta}(Z(\lambda)) = \mu_{y,\beta}(Z(\nu))\]
 for any $\R_+$-functor $y$ and $\beta \in (0,\infty)$.  The fact that $s(\lambda) = s(\nu)$ (and hence $\xi^{y,\beta}_{s(\lambda)} = \xi^{y,\beta}_{s(\nu)} > 0$) whenever $(\lambda,\nu) \in P_\Lambda$ then implies that
 \[ e^{\beta y(\lambda)} \rho(B(y,\beta))^{d(\lambda)} = e^{\beta y(\nu)} \rho(B(y,\beta))^{d(\nu)}.\]
 Taking logarithms of both sides and dividing by $\beta$ yields the left-hand side of Equation \eqref{eq:needed-for-aperiodicity}, yet $d(\lambda) \not= d(\nu)$.
\end{proof}

The preceding Corollary generalizes \cite[Corollary 7.2]{aHLRS}, which establishes that for periodic $k$-graphs which are coordinatewise irreducible, the set $\{ \ln(\rho(A_i))\}_{i=1}^k$ is rationally dependent.
Indeed, rational dependence of the set $\{ \ln(\rho(A_i))\}_{i=1}^k$ implies that Equation \eqref{eq:needed-for-aperiodicity} fails for  $y = 0$ and $\beta =1$.  Corollary \ref{cor:aperiodicity} implies that for periodic $k$-graphs which are coordinatewise irreducible, Equation \eqref{eq:needed-for-aperiodicity} must fail for all choices of $y$ and $\beta$.
In other words, Corollary \ref{cor:aperiodicity} offers an expanded set of strategies for detecting aperiodicity of the $k$-graph, and hence the simplicity of $C^*(\Lambda)$.

\section{Examples of $\R_+$-functors}

\label{sec:examples}
Before addressing the relationship between $\R_+$-functors and Hausdorff measures in Section \ref{sec:hausdorff}, we pause to discuss the range of possibilities for $\R_+$-functors on two examples of finite, strongly connected 2-graphs.  We also describe the associated actions and quasi-invariant measures.  Finally, we identify the periodicity groups of these 2-graphs; Theorem \ref{thm:KMS-simplex} then enables us to reconstruct the KMS states associated to these 2-graphs and $\R_+$-functors.

\subsection{An example from \cite{McNamara}}
\label{sec:McNamara}

We begin with an example which was studied in Section 5.11 of McNamara's thesis \cite{McNamara}.  Namely, $\Lambda$ is a 2-graph with one vertex $v$ and two blue edges (called $e_1, e_2$) and two red edges (called $f_1, f_2$) and factorization relations
\[
e_1f_1 = f_1e_1,\ e_1f_2 = f_1e_2,\ e_2f_1 = f_2e_1,\  e_2f_2 = f_2e_2.
\]
Consequently, $Z(e_i) = Z(f_i)$ for $i =1,2$.  It follows that $\Lambda$ is periodic and $\text{Per}\, (\Lambda) \supseteq \Z\{(1,-1)\}$.  Indeed, this inclusion is an equality: if $\text{Per}\, (\Lambda) \supsetneq \Z\{(1,-1)\}$, then there would exist integers $\{n_i, m_i\}_{i=1}^\ell$ such that
\[ Z(v) = Z(f_1^{n_1} f_2^{m_1} \cdots f_1^{n_\ell} f_2^{m_\ell}).\]
Since $Z(f_1) \sqcup Z(f_2) = Z(v)$, however, this is impossible.

To describe all the $\R_+$-functors $y$ on this 2-graph, set
${\bf{e}}_i=y(e_i)$, and ${\bf f}_i=y(f_i)$.
The factorization relations then tell us that all the $\R_+$-functors must  satisfy
\[
{\bf e}_1+{\bf f}_2 = {\bf f}_1+{\bf e}_2,\ {\bf e}_2+{\bf f}_1 = {\bf f}_2+{\bf e}_1.
\]
In other words, we have 3 free variables ${\bf e}_1, {\bf e}_2, {\bf f}_1$; and ${\bf f}_2 = {\bf f}_1 + {\bf e}_2 - {\bf e}_1$.
Since $|\Lambda^0| =1$, we have
\begin{align*}
B_1(y, \theta)&=\begin{pmatrix} e^{-\theta {\bf e}_1}+ e^{-\theta {\bf e}_2}\end{pmatrix},\\
\ B_2(y, \theta)&=\begin{pmatrix} e^{-\theta {\bf f}_1}+ e^{-\theta {\bf f}_2}\end{pmatrix}=\begin{pmatrix} e^{-\theta {\bf f}_1}+ e^{-\theta ({\bf f}_1+{\bf e}_2-{\bf e}_1)},\end{pmatrix}
\end{align*}
and $\xi^{y,\theta} =(1)$ for any $y, \theta$.

The fact that $\Lambda$ has only one vertex implies that each matrix $B_i(y,\beta)$ will be irreducible, for any choice of $y$ and $\beta$. By Corollary \ref{cor:aperiodicity},
the periodicity of $\Lambda$ means that Equation \eqref{eq:needed-for-aperiodicity} will never hold, regardless of our choice of $y$ and $\beta$.  Indeed, although $d(f_1) \not= d(e_1)$, we always have
\[ y(e_1) + \frac{1}{\beta} \ln(\rho(B_1(y,\beta)) = y(f_1) + \frac{1}{\beta} \ln(\rho(B_2(y,\beta)).\]
Consequently, since
\[ \alpha^{y,\beta}_t(s_\lambda) = e^{it(y(\lambda) + \frac{1}{\beta} \ln(\rho(B(y,\beta))^{d(\lambda)}))} s_\lambda,\]
the action $\alpha^{y,\beta}$ scales both $s_{e_1}$ and $s_{f_1}$ by the same complex number.

Recall  that
 every infinite path in $\Lambda^\infty$ can be written uniquely as a one-sided infinite sequence of edges which alternate blue-red-blue-red.
 In $\Lambda$, all edges are composable, so $\Lambda^\infty$ is naturally homeomorphic to
the infinite product $\prod_{\N}\{ 0,1\}$.
Moreover, $\mu_{y,\theta} (Z(e_2))=1- \mu_{y,\theta} (Z(e_1))$ because $\Lambda^\infty= Z(e_1) \sqcup Z(e_2)$. By our identification of $\Lambda^\infty$ with $\prod_\N \{ 0,1\}$, we can view $\mu_{y,\theta}$ as a Markov measure $\mu_x$ on $\prod_\N \{ 0,1\}$, where $x=\mu_{y,\theta} (Z({e}_1))=\mu_{y,\theta}(Z(f_1))$.
In the notation of Section 3.1 of \cite{dutkay-jorgensen-monic}, the  Markov measure $\mu_x$ on $\prod_{n\in \N} \{0,1\}$ corresponds to the matrix
\[
T_x=\begin{pmatrix} x & 1-x\\  1-x & x\end{pmatrix},\]
and assigns measure $\mu_x(Z(a_1 \cdots a_n)) = x^{\#\{i: a_i = 0\}} (1-x)^{\#\{ j: a_j =1\}}$ to the cylinder set $Z(a_1 \cdots a_n)$.
Indeed, \cite[Theorem 3.9]{dutkay-jorgensen-monic} implies that if $x \not= x'$ then $\mu_x$ and $\mu_{x'}$ are mutually singular.  It follows that for $x =\mu_{y,\theta}( Z({e}_1)) \not= 1/2$, the measure $\mu_{y,\theta}$ is mutually singular with respect to the measure $M$ from \cite[Proposition 8.1]{aHLRS}.

In fact, the correspondence taking $(y,\theta)$ to $x$ such that $\mu_{y,\theta}= \mu_x$ is surjective.  That is,
 given $x \in (0, 1/2)$, we will describe a way to choose a pair $(y,\theta)$ such that  $\mu_{y,\theta}(Z(e_1)) = \mu_x$.
Having chosen $x \in (0,1/2)$ and $y(e_1)$, choose $\theta > \frac{ \ln((1-x)/x) }{y(e_1)} > 0$ and define
\[ y(e_2)= y(e_1) - \frac{1}{\theta} \ln \left( \frac{1-x}{x} \right).\]
Note that $y(e_2)$ will be positive whenever  $\theta > \frac{ \ln((1-x)/x) }{y(e_1)}$.
Setting $y(f_i) := y(e_i)$ completes the definition of the $\R_+$-functor.  However, other constructions of weight functors are also possible; cf.~\cite[Example 5.31]{McNamara}.

\begin{rmk} This example can be extended to the setting of 2-graphs with one vertex and $N$ edges of each color, using the Markov measures associated to $N \times N$ matrices from \cite{dutkay-jorgensen-monic}.
	\end{rmk}

\begin{rmk}
For a fixed $\R_+$-functor $y$, Corollary \ref{cor:continuity-of-measures} tells us that varying $\theta$ produces a continuous family of measures $\mu_{y,\theta}$.  However, if $x \not= x'$, the Markov measures $\mu_x$ and $\mu_{x'}$ are mutually singular.  Thus, equivalence and continuity of a family of measures are different concepts.
\end{rmk}

\subsection{An example from \cite{LLNSW}}
\label{sec:eyeglasses}
 Another motivating example for us was  the 2-graph of \cite[Example 7.7]{LLNSW}, which is described by the edge-colored directed graph
 
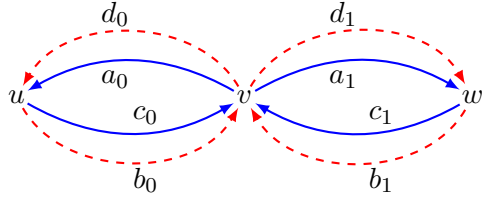
\[
\begin{tikzpicture}[scale=2]
\node[inner sep=0.5pt, circle] (u) at (0,0) {$u$};
\node[inner sep=0.5pt, circle] (v) at (1.5,0) {$v$};
\node[inner sep=0.5pt, circle] (w) at (3,0) {$w$};
\draw[-latex, thick, blue] (v) edge [out=150, in=30] (u); %a_0
\draw[-latex, thick, blue] (u) edge [out=-30, in=210] (v); %c_0
\draw[-latex, thick, blue] (v) edge [out=30, in=150] (w); %a_1
\draw[-latex, thick, blue] (w) edge [out=210, in=-30] (v); %c_1
\draw[-latex, thick, red, dashed] (v) edge [out=120, in=60] (u); %d_0
\draw[-latex, thick, red, dashed] (u) edge [out=-60, in=240] (v); %b_0
\draw[-latex, thick, red, dashed] (v) edge [out=60, in=120] (w); %d_0
\draw[-latex, thick, red, dashed] (w) edge [out=240, in=-60] (v); %b_0
\node at (0.65, 0.12) {\color{black} $a_0$};
\node at (0.85, -0.12) {\color{black} $c_0$};
\node at (2.15, 0.12) {\color{black} $a_1$};
\node at (2.4, -0.12) {\color{black} $c_1$};
\node at (0.65, 0.55) {\color{black} $d_0$};
\node at (0.85, -0.55) {\color{black} $b_0$};
\node at (2.15, 0.55) {\color{black} $d_1$};
\node at (2.4, -0.55) {\color{black} $b_1$};
\end{tikzpicture}
\]
with factorization rules
\begin{align*}
a_0b_0 &= d_0c_0 \qquad a_1 b_1 = d_1 c_1 \qquad a_1b_0 = d_1c_0\\
a_0 b_1 &= d_0 c_1\qquad c_1 d_1 = b_0 a_0\qquad
c_0d_0 = b_1a_1.
\end{align*}
Again, for any edge $f$, write ${\bf f}$  for the value  $y(f)$.

The linear system arising from the factorization relations that a $\R_+$-functor on $\Lambda$ must satisfy consists of 6 equations, which we write in compressed form as
\begin{align*}
{\bf a}_i+{\bf b}_i &= {\bf d}_i+{\bf c}_i,\,\, i=0,1\\
{\bf a}_i+{\bf b}_i &= {\bf d}_i+{\bf c}_{1-i}, \,\,i=0,1\\
{\bf a}_i+{\bf b}_i &= {\bf d}_{1-i}+{\bf c}_{1-i}, \,\,i=0,1,\\
\end{align*}
This system has 4 free variables ($\bf b_1,\bf  c_1,\bf d_0,\bf d_1$).

Now, suppose we have chosen a $\R_+$-functor $y$ on $\Lambda$ and $\theta \in (0, \infty)$.
Define, for $i = 0,1$,
\begin{align*}
A_i &:= e^{-\theta \, {\bf a}_i}\quad
B_i :=e^{-\theta \, {\bf b}_i}\quad
C_i :=e^{-\theta \, {\bf c}_i}\quad
D_i := e^{-\theta \, {\bf d}_i};
\end{align*}
\[
\text{then} \quad { B_1(y, \theta) = \begin{pmatrix} 0& A_0
	&0 \\ C_0&0&C_1 \\ 0&A_1&0\end{pmatrix}}, \ { B_2(y, \theta) = \begin{pmatrix} 0& D_0
	&0 \\ B_0&0&B_1 \\ 0&D_1&0\end{pmatrix}}.
\]
Straightforward computations reveal that
\[
{\rho(B_1(y, \theta))= \sqrt{A_0C_0+A_1C_1}  = \sqrt{2} \, C_1\, \Big(\frac{D_1}{ B_1}\Big)^{1/2}}
\]
and the unimodular positive eigenvector for $B_1(y,\theta)$ is
\[
{
 \frac{1}{A_0+A_1 + \rho(B_1(y,\theta))} \, \begin{pmatrix} A_0 \\ \rho(B_1(y, \theta)) \\ A_1  \end{pmatrix}}.
\]
Similarly, $
{\rho(B_2(y, \theta))= \sqrt{B_0D_0+B_1D_1} = \sqrt{2}\, (B_1\,D_1)^{1/2}},
$
and the unimodular positive eigenvector for $B_2(y,\theta)$ is
\[
{
	\frac{1} {D_0+D_1+\rho(B_2(y,\theta))}\, \begin{pmatrix} D_0 \\  \rho(B_2(y, \theta)) \\ D_1 \end{pmatrix}}.
\]
Lemma \ref{lemma: Lambda strongly connected} implies that $B_1(y,\theta)$ and $B_2(y,\theta)$ have a unique common positive unimodular eigenvector $\xi^{y,\theta}$, so the  eigenvectors for $B_1(y,\theta)$ and $B_2(y,\theta)$ must be equal.
Moreover, we have
\[
{ \rho(B_1(y, \theta))} {\rho(B_2(y, \theta))}
=2\, C_1 D_1=2\, A_0 B_0.
\]

With the above information,  we can now compute the probability measure $\mu_{y,\theta}$ on some cylinder sets $Z(\lambda)$.  First, observe that
\begin{align*}
\mu_{y,\theta}(Z(a_0b_0))&= A_0\,B_0
{(\rho(B_1(y, \theta)))^{-1}} {(\rho(B_2(y, \theta)))^{-1}}  \xi^{y,\theta}_{s(a_0b_0)} \\
&=\frac12\, \xi^{y,\theta}_{u} .
\end{align*}

Proposition 4.6 of \cite{FGJorKP} explains how a matrix $T_x = \begin{pmatrix}
x & 1-x \\ 1-x & x
\end{pmatrix}$, for $x \in (0,1)$, can be used to construct Markov measures on $\Lambda^\infty$.  For different  values of $x$, the associated Markov measures are inequivalent.  We observe also that the Markov measures  $\mu_x$ from \cite[Proposition 4.6]{FGJorKP} are not probability measures; rather, $\mu_x(\Lambda^\infty) = 2$ .

However, the measure $\mu_{y,\theta}$ can only be a (rescaled) Markov measure for $x=1/2$.  To see this, we recall from \cite{FGJorKP} that
\[ \mu_x(Z(a_0 b_0)) = T_x(1,1) T_x(1,1) = x^2, \quad \text{ while } \quad \mu_x(Z(u)) = T_x(1,1) = x.\]
Therefore, if $\mu_{y,\theta} = 1/2 \mu_x$ for some $x$, then
\[\frac{1}{2} =  \frac{\mu_{y,\theta}(Z(a_0 b_0)}{\mu_{y,\theta}(Z(u))} = \frac{\mu_x(Z(a_0 b_0))}{\mu_x(Z(u))} = x.\]
Note that the Markov measure $\mu_{1/2}$ also assigns $\mu_{1/2}(Z(w)) = 1/2$, and $\mu_{1/2}(Z(v)) = 1$.

Now, recall  from \cite{FGJorKP} that $\mu_{1/2} = 2 M$ where $M$ denotes the measure from \cite[Proposition 8.1]{aHLRS}.  This measure $M$ also arises as $\mu_{y,\theta}$ when  $y=0$. 
In other words, the only way that $\mu_{y,\theta}$ can be a (rescaled) Markov measure is if $\mu_{y,\theta} = M$.

We can completely characterize the KMS$_\beta$ states of $(C^*(\Lambda), \alpha^{y,\beta})$. Indeed, by \cite[Example 7.7]{LLNSW}, we know that $\text{Per}(\Lambda)=2\Z(1,-1)$. Theorem~\ref{thm:KMS-simplex} therefore implies that the simplex of KMS$_\beta$ states is isomorphic to the tracial state space of $C(\mathbb{T})$.

\section{Weights, ultrametrics, and Hausdorff structure}
\label{sec:hausdorff}

In this section, we use the same data (an $\R_+$-functor and a positive number $\theta$)  that we employed in Section \ref{sec:kms} to define the generalized gauge action $\alpha^{y,\theta}$ for a different purpose: namely, we construct  an ultrametric $d_{y,\theta}$ on the infinite path space $\Lambda^\infty$, which we view as a  Cantor set.
 We then compute the Hausdorff dimension and Hausdorff measure of $(\Lambda^\infty, d_{y,\theta})$: Corollary \ref{cor:Hausdorff-dim} establishes that the Hausdorff dimension of $(\Lambda^\infty, d_{y,\theta})$ is $\theta$ -- the same as the inverse temperature for which we characterized the KMS states for the associated dynamics $\alpha^{y,\theta}$ of $C^*(\Lambda)$ in Theorem \ref{thm:KMS-simplex} -- and the associated Hausdorff measure is precisely our unique quasi-invariant measure $\mu_{y,\theta}$.
 In fact, we prove a result about Hausdorff dimension in a greater generality involving weights on Bratteli diagrams with a certain self-similarity property, see Theorem \ref{thm:Hausdorff-dim-self-sim}.

The examples of $k$-graphs and $\R_+$-functors which we discussed in  Section \ref{sec:examples} satisfy the hypotheses of Corollary \ref{cor:Hausdorff-dim}.   In particular, the 2-graph of Section \ref{sec:McNamara} admits $\R_+$-functors giving rise to a large family of inequivalent Hausdorff measures on $\Lambda^\infty$.

Following \cite{pearson-bellissard, julien-savinien-transversal}, our ultrametrics $d_{y,\theta}$ are constructed using weights on Bratteli diagrams.   Thus, we begin by reviewing the construction of a Bratteli diagram associated to a higher-rank graph from \cite{FGJKP2} (Definition \ref{def-k-brat-diagrm}) and discussing how to use a $\R_+$-functor to construct weights on the Bratteli diagram (Propositions \ref{prop:weight} and \ref{prop:weight-one-vertex}).   We then show, in Proposition \ref{prop:ultrametric},  that   the ultrametric $d_{y,\theta}$ arising from such a weight metrizes the cylinder set topology on $\Lambda^\infty$.

We note that weights on Bratteli diagrams and the associated ultrametrics have been studied by many authors \cite{pearson-bellissard, julien-savinien-transversal, FGJKP2}. In particular, Pearson and Bellissard \cite{pearson-bellissard} were motivated by work of Michon \cite{mich}, who  introduced the notion of a weighted tree in his study of Gibbs measures on Cantor sets; see also \cite{julien-savinien-selfsim}.

\subsection{Defining weights and metrics on Bratteli diagrams}
\label{sec:Bratteli-defns}
\begin{defn}[\cite{FGJKP2} Definition 2.5]
\label{def-k-brat-diagrm}
Let $\Lambda$ be a finite $k$-graph with coordinate matrices $A_1, \ldots, A_n$. The \emph{stationary $k$-Bratteli diagram} associated to $\Lambda$, which we will call  $ \mathcal{B}_{\Lambda}$, is given by a filtered set of vertices $\mathcal{V} = \bigsqcup_{n\in \N} \mathcal{V}_n$ and a filtered set of edges $\mathcal{E} = \bigsqcup_{n\geq 1} \mathcal{E}_n$, where the edges in $\mathcal{E}_n$ go from $\mathcal{V}_{n}$ to $\mathcal{V}_{n-1}$, such that:
\begin{itemize}
\item[(a)] For each $n \in \N$, $\mathcal{V}_n = \Lambda^0$ consists of the vertices of $\Lambda$.

\item[(b)]  When $ n \equiv i \pmod{k}$, there are $A_i(p,q)$ edges whose range is the vertex $p$ of $\mathcal{V}_{n-1}$ and whose source is the vertex $q$ of $\mathcal{V}_{n}$.
\end{itemize}
A {\em path} (finite or infinite) in the Bratteli diagram $\mathcal B_\Lambda$ is a path with range in $\mathcal V_0$.  We write $|\eta|$ for the length (number of edges) of a finite path $\eta$ in the Bratteli diagram, and $F^n\mathcal B_\Lambda$ for the finite paths of length $n$.
We also write
\[F\mathcal B_\Lambda := \bigcup_n F^n\mathcal B_\Lambda.\]
\end{defn}
Proposition 2.10 and Remark 2.11 of \cite{FGJKP2} discuss the relationship between paths (finite and infinite) in $\Lambda$ and $\mathcal B_\Lambda$.  In particular, every finite path in $\mathcal B_\Lambda$ is represented by a string of composable edges in $\Lambda$.  Consequently, if $\eta \in F\mathcal B_\Lambda$ we will also write $\eta \in \Lambda$ to denote the unique morphism in $\Lambda$ represented by the string $\eta$ of composable edges. However, not every finite path in $\Lambda$ corresponds to a finite path in $\mathcal B_\Lambda$.  For example, a path in $\Lambda$ consisting of two red edges will not occur in $\mathcal B_\Lambda$.

The space of infinite paths in $\mathcal B_\Lambda$ is also denoted the {\em boundary} of the Bratteli diagram in some references.  It is canonically equipped with the cylinder set topology, whose basic open sets are $\{ Z(\lambda)\}_{\lambda \in F\mathcal B_\Lambda}$, where
$Z(\lambda)$ is the set of infinite paths whose initial segment is $\lambda$.
Proposition 2.10 of \cite{FGJKP2} shows that when we equip both spaces with the cylinder set topology, $\Lambda^\infty$ is homeomorphic to the space of infinite paths in $\mathcal B_\Lambda$.

The following is a slight modification of the definition of a weight on a Bratteli diagram as introduced in, for example, \cite{pearson-bellissard, julien-savinien-transversal, FGJKP2}.
Although we state Definition \ref{def:weight} for arbitrary Bratteli diagrams, in this paper we will  apply  it mainly to stationary $k$-Bratteli diagrams.

\begin{defn}[compare with \cite{pearson-bellissard} Definition 8, \cite{julien-savinien-transversal} Definition 2.9, \cite{FGJKP2} Definition 2.14]
\label{def:weight}
 A \emph{weight} on a Bratteli diagram $\mathcal{B}$ is a function $w: F\mathcal{B} \to \R_+$ such that
\begin{enumerate}
\item[(i)] $w(v)\leq 1$ for all $v \in \mathcal{V}_0$.
\item[(ii)]  $\lim_{n\to \infty} \sup \{ w(\gamma) \mid \gamma \in F^n\mathcal{B} \} = 0$.
\item[(iii)] If $\eta$ is a sub-path of $\gamma$, then $w(\gamma) \leq  w(\eta)$.
\end{enumerate}
\end{defn}
In this paper,  we work primarily with Bratteli diagrams associated to finite $k$-graphs.  Under these hypotheses, the supremum in  condition (ii) above is actually a maximum.

\begin{defn}
\label{def:ultrametric}
A metric $d$ on a space $X$ is an {\em ultrametric} if
\begin{equation}
\label{eq:ultrametric}
d(x,y) \leq \max \{ d(x,z), d(y,z)\mid  z \in X\}.
\end{equation}
\end{defn}

The following Proposition shows that the first part of the conclusion of
Proposition 2.15 of \cite{FGJKP2} still holds with our revised definition of a weight.  The second part of that proposition, asserting that the ultrametric topology agrees with the cylinder set topology, need not hold in general but it does hold in our case of interest; see Proposition \ref{prop:ultrametric} below.

\begin{prop}
Let $w$ be a weight on a Bratteli diagram $\mathcal B$.  The formula
\begin{equation}
d_w(x,z) = \begin{cases}
1, & r(x) \not= r(z)\\
0, & x=z \\
w(x \wedge z), &\text{ else.}
\end{cases}
\label{eq:distance-from-weight}
\end{equation}
defines an ultrametric on the space $X_{\mathcal B}$ of infinite paths in $\mathcal B$.
Here $x \wedge z \in F\mathcal B_\Lambda$ denotes the longest common initial segment of $x$ and $z$.
\label{prop:weight-gives-ultrametric}
\end{prop}
\begin{proof}
This follows verbatim from the first part of the proof of \cite[Proposition 2.15]{FGJKP2}.
\end{proof}

The following Lemma establishes conditions under which the hypotheses  of Proposition \ref{prop:weight} are satisfied.  These conditions are not necessary; for example, the coordinate matrices for the 2-graph studied in Section \ref{sec:eyeglasses} satisfy the conclusion of Lemma \ref{lem-sp-rad-entries} (and thus the hypotheses of Proposition \ref{prop:weight}) but not the hypotheses of Lemma \ref{lem-sp-rad-entries}.  
\begin{lemma}
\label{lem-sp-rad-entries}
Let $B$ be a nonnegative 
 matrix with at least two non-zero entries per row.
Then the spectral radius of $B$ is strictly greater than any of the entries of $B$.
\end{lemma}

\begin{proof} Write $R$ for the positive square root of $BB^t$, and
 notice that $R$ is  Hermitian.  By
\cite[Theorem 2]{schw} we have that
\[
\rho(R) \leq \rho( B) .
\]
By the spectral theorem, it also follows that
\[
\rho(R)^2= \rho(R^2)=\rho( B B^t) .
\]
Now assume that $m =B_{q,r}$ is the largest entry of $B$; this implies that the $(q,q)$-entry of
$B B^t $ is strictly greater than $m^2$. Now by using Rayleigh quotients to bound the spectral radius for $B B^t $, we get, if we denote by  $e_j$ the standard basis for $\R^n$, that
\[
\rho(R^2) \geq (B B^t e_q, e_q ) > m^2\ \Rightarrow\  \sqrt{\rho(R^2)}  > m.
\]
It follows that $
m < \rho(R) \leq \rho( B),
$
as desired.
\end{proof}

In preparation for the next two propositions, we first note that due to the fact that every path in $F\mathcal B_\Lambda$ is given by a string of composable edges in $\Lambda$, and hence represents a unique morphism in $\Lambda$, cf.~\cite[Remark 2.11]{FGJKP2}, we can (and will) interpret a $\R_+$-functor $y$ also as an additive functor $y: F\mathcal B_\Lambda \to \R_+$. Second, we identify a necessary condition on the matrices $B_i(y,\theta)$ which ensures that we obtain a weight on $F\mathcal{B}_{\Lambda}$. Since the condition will differ in the one-vertex case and the general finite $k$-graph case, we list it for easy reference in the following two formulations:
\begin{enumerate}
\item[(w-I)] $\forall i=1,\dots, k$, $\rho(B_i(y, \theta))>1$.

\item[(w-II)] $\exists i=1,\dots ,k$ such that $\rho(B_i(y,\theta)) > \max\{B_i(y,\theta)_{v,w} \mid v, w \in \Lambda^0\}$.
\end{enumerate}
Observe that condition (w-I) is satisfied for all $\theta$ near 0 if  the $k$-graph has at least two edges of every color; the   example of Section \ref{sec:McNamara} satisfies this condition.

\begin{prop}
\label{prop:weight}
 Let $\Lambda$ be a finite, strongly connected $k$-graph with at least two vertices.
Let $y$ be a $\R_+$-functor on $\Lambda$ and $\theta \in \R_+$
such that condition (\text{w}-II) holds. 
Then the function $w_{y,\theta}: F\mathcal{B}_{\Lambda} \to \R_{>0}$ given by
\begin{equation}
\label{eq:weight-y}
w_{y, \theta}(\lambda) = e^{-y(\lambda)} \left( \rho(B(y, \theta))^{-d(\lambda)} \xi^{y,\theta}_{s(\lambda)}\right)^{1/\theta}\end{equation}
is a weight on $\mathcal{B}_{\Lambda}$.
\end{prop}

\begin{proof}
Since $\xi^{y,\theta} \in (\R^+)^{\Lambda^0}$ has $\ell^1$-norm 1, $w_{y,\theta}$ satisfies  the first condition  of Definition \ref{def:weight}.   We next check the third condition. Let $\eta \in F\mathcal B_\Lambda$ be a finite path; suppose that $|\eta| =qk +(i-1)$, so that every edge extending $\eta$  will have degree $e_i$.  Writing $w := s(\eta)$, we compute:
\begin{align*}
\rho(B(y, \theta))^{-d(\eta)} \xi^{y,\theta}_{w} &= \rho(B(y,\theta))^{-d(\eta) - e_i}  \sum_{v \in \Lambda^0} B_i(y, \theta)_{w, v} \xi^{y,\theta}_v \\
&= \rho(B(y,\theta))^{-d(\eta)-e_i} \sum_{\lambda \in w\Lambda^{e_i}} e^{-y(\lambda) \theta} \xi^{y,\theta}_{s(\lambda)}.
\end{align*}
Moreover,  each summand is strictly positive, and therefore
\[ \rho(B(y, \theta))^{-d(\eta)} \xi^{y,\theta}_{w}\geq  \rho(B(y,\theta))^{-d(\eta) - e_i} e^{-y(\lambda) \theta} \xi^{y,\theta}_{s(\lambda)}\]
for any $\lambda \in \mathcal E_{qk+i} = \Lambda^{e_i}$. Hence,
\[\left( \rho(B(y, \theta))^{-d(\eta)} \xi^{y,\theta}_{w} \right)^{1/\theta} \geq e^{-y(\lambda) }  \left(\rho(B(y,\theta))^{-d(\eta) - e_i} \xi^{y,\theta}_{s(\lambda)}\right)^{1/\theta} \]
for any such $\lambda$.  It follows that, given
 any finite path $\eta \in F\mathcal B_{\Lambda}$ and any extension $\eta \lambda$ of $\eta$,
 \begin{align*}
 e^{-y(\eta)} \left( \rho(B(y, \theta))^{-d(\eta)} \xi^{y,\theta}_{s(\eta)} \right)^{1/\theta} &\geq   e^{-y(\eta) } e^{-y(\lambda) }  \left(\rho(B(y,\theta))^{-d(\eta \lambda)} \xi^{y,\theta}_{s(\lambda)}\right)^{1/\theta} .
 \end{align*}
The additivity of $y$ thus implies that
\begin{equation}
\label{eq:weight-strictly-decreasing}
w_{y,\theta}(\eta \lambda) \leq w_{y,\theta}(\eta),\end{equation} so the third condition of Definition \ref{def:weight} is satisfied.

For the second condition, first note that our calculations above imply that
 \[\sup \{ w_{y,\theta}(\gamma) \mid \gamma \in F^{n+1}\mathcal B_{\Lambda}\} \leq \sup \{ w_{y,\theta}(\gamma) \mid \gamma \in F^n \mathcal B_{\Lambda}\}.\]
Moreover,  for any non-negative matrix $B$, \cite{schw} implies that
\[ \rho(B)^2 \geq \rho(BB^t).\]
The fact that $\Lambda$ is strongly connected, and hence source-free by \cite[Lemma 2.1]{aHLRS}, implies that the matrix $B_i(y,\theta)$ has a nonzero entry in each row. Therefore, every diagonal of $B_i(y,\theta) B_i(y,\theta)^t$ is nonzero, and
$\rho(B_i(y,\theta)B_i(y,\theta)^t) \geq \max_{v \in \Lambda^0} \{ B_i(y,\theta)_{vv}^2\}.$
It follows that
\[ \rho_i := \rho(B_i(y,\theta)) \geq \max_{v\in \Lambda^0} B_i(y,\theta)_{vv}.\]

We furthermore recall that, for each $1 \leq j \leq k$;  $v, w \in \Lambda^0$; and  $f \in v \Lambda^{e_j} w$, we have 
$e^{-y(f)\theta} \leq B_j(y,\theta)_{vw}.$
 It follows that 
\begin{equation}
\label{eq:bound-on-sup}
\sup \{e^{-y(f)}\mid  d(f) = e_j\} \leq  \left( \sup \{ B_j(y,\theta)_{vw} \mid  v, w \in \Lambda^0\}\right)^{1/\theta} \leq \rho_j^{1/\theta}.\end{equation}
 Writing $n =qk+t$ for $0 \leq t \leq k-1$, the sequence 
\begin{equation} \left( \sup\{ e^{-y(\lambda)} \rho(B(y,\theta))^{-d(\lambda)/\theta} \mid \lambda \in F^n\mathcal B_\Lambda \} \right)_{n\in \N}
\label{eq:bound-for-weight}
\end{equation}
tends to zero iff $w_{y,\theta}$ satisfies the second condition of a weight, because
the fact that $\Lambda$ is finite implies that the set $\{ (\xi^{y,\theta}_v)^{1/\theta}: v \in \Lambda^0\}$ is bounded (and bounded away from zero).

Since each term in this sequence is bounded by 1, our assumption that $\sup\{ B_i(y,\theta)_{vw} \mid v, w \in \Lambda^0\} < \rho_i$ for at least one $i$, combined with Equation \eqref{eq:bound-on-sup}, forces the sequence \eqref{eq:bound-for-weight} to tend to zero.
 Consequently, the sequence $\left( \sup\{ w_{y,\theta}(\gamma) \mid \gamma \in F^n \mathcal B_{\Lambda}\}\right)_{n\in \N}$ -- being bounded above by the product of the  sequence \eqref{eq:bound-for-weight} and the maximum of $\{ (\xi^{y,\theta}_v)^{1/\theta}\}_{v \in \Lambda^0}$ -- also tends to zero as $n\to \infty$.
 \end{proof}

Before stating the following Proposition, we remind the reader that if $|\Lambda^0|=1$, then the unimodular Perron-Frobenius eigenvector $\xi^{y,\theta} \in \left( \R_{> 0}\right)^{\Lambda^0}$ must be the constant vector $(1)$.  With this in mind, Equations \eqref{eq:weight-y} and \eqref{eq:weight-y-one-vertx} give the same formula in the case of one-vertex $k$-graphs.
 \begin{prop}
\label{prop:weight-one-vertex}
 Let $\Lambda$ be a finite, strongly connected one-vertex $k$-graph.
 Suppose that an $\R_+$-functor $y$ and  $\theta \in \R_{> 0}$ have been chosen such that condition (\text{w}-I) holds. Then the function $w_{y,\theta}: F\mathcal{B}_{\Lambda} \to \R_{>0}$ given by
	\begin{equation}
	\label{eq:weight-y-one-vertx}
	w_{y, \theta}(\lambda) = e^{-y(\lambda)} \left( \rho(B(y, \theta))^{-d(\lambda)} \right)^{1/\theta}\end{equation}
	is a weight on $\mathcal{B}_{\Lambda}$.
\end{prop}

\begin{proof}
  First notice that condition (i) of Definition \ref{def:weight} holds immediately; $w_{y,\theta}(v) = 1$ for the unique $v \in \Lambda^0$.  Condition (iii) follows immediately from condition (w-I).
	To check condition (ii) of Definition \ref{def:weight}, we simply observe that, again thanks to condition (w-I), that
	$ \left( \rho(B(y, \theta))^{-d(\lambda)} \right)^{1/\theta}\to 0$ as $d(\lambda) \to \infty.$
\end{proof}

The next Lemma establishes the crucial condition of our weights $w_{y,\theta}$, which guarantees that the associated ultrametric $d_{y,\theta}$  metrizes the cylinder set topology on $X_{\mathcal B_\Lambda} \cong \Lambda^\infty$, the infinite path space of $\Lambda$.  (However, Lemma \ref{lem-weights-add-up} does not actually require that the function $w_{y,\theta}$  be a weight.) We will also rely on Lemma \ref{lem-weights-add-up} to prove Corollary \ref{cor:Hausdorff-dim}.  

\begin{lemma}\label{lem-weights-add-up} Let $\Lambda$ be a finite and strongly connected $k$-graph,   $y$ a  $\R_+$-functor  on $\Lambda$ and   $\theta \in (0,\infty)$.  Let $w_{y,\theta}$ denote the function from Equation \eqref{eq:weight-y}.
For any finite path $\lambda \in \mathcal B_{\Lambda}$, and any $m\in \N$, 
\begin{equation}\label{eq:self similar weight}
	w_{y,\theta}(\lambda)^\theta  = \sum_{\lambda \eta \in F^{|\lambda| + m}\mathcal B_{ \Lambda}} w_{y,\theta}(\lambda \eta)^\theta.
\end{equation}

	\end{lemma}
\begin{proof} 
	The fact that $\xi^{y,\theta}$ is an eigenvector for  each matrix $B_i(y,\theta)$ with eigenvalue $\rho_i$ implies that, for any path $\lambda\eta \in F\mathcal B_\Lambda$,
	\begin{align*}
	w_{y,\theta}(\lambda)^\theta &= e^{-\theta y(\lambda)} \rho(B(y,\theta))^{-d(\lambda)} \xi^{y,\theta}_{s(\lambda)} \\
	&= e^{-\theta y(\lambda)} \rho(B(y, \theta))^{-d(\lambda) - d(\eta)} \sum_{v \in \Lambda^0} B (y, \theta)^{d(\eta)}_{s(\lambda), v} \xi^{y, \theta}_{v} \\
	&= e^{-\theta y(\lambda)} \rho(B(y, \theta))^{-d(\lambda) - d(\eta)} \sum_{\lambda \tilde \eta \in F^{|\lambda \eta|} \mathcal B_\Lambda} e^{-\theta y(\tilde \eta)} \xi^{y,\theta}_{s(\tilde \eta)}
	\\
	&=\sum_{\lambda \tilde \eta \in F^{|\lambda \eta|} \mathcal B_\Lambda} e^{-\theta y(\lambda\tilde \eta)} \rho(B(y, \theta))^{-d(\lambda) - d(\eta)}   \xi^{y,\theta}_{s(\tilde \eta)}\\
	& =\sum_{\tilde \eta: \lambda \tilde \eta \in F^{|\lambda \eta|} \mathcal B_\Lambda} w_{y,\theta}(\lambda \tilde \eta)^\theta.
	\end{align*}
	Furthermore, since $\Lambda$ is strongly connected, there is a path $\lambda \eta \in F^{|\lambda| + m} \mathcal B_\Lambda$ for any $m \in \N$.
	Since $\eta$ was arbitrary, this finishes the proof.
	\end{proof}

\begin{prop}
\label{prop:ultrametric}
Let $\Lambda$ be a finite, strongly connected $k$-graph.  Suppose that  $\theta \in \R_+$ and a $\R_+$-functor $y$ on $\Lambda$  exist
such that condition (\text{w}-I) holds if $\vert \Lambda^0\vert=1$ and otherwise  condition (\text{w}-II) holds.   Then the formula $d_{y,\theta}$ associated to $w_{y,\theta}$ as in Proposition \ref{prop:weight-gives-ultrametric},
\begin{equation}
\begin{split}
d_{y,\theta} (x,z)& = \inf \left\{ e^{-y(\lambda)} \left( \rho(B(y,\theta))^{-d(\lambda)} \xi^{y,\theta}_{s(\lambda)} \right)^{1/\theta} \mid x, z \in Z(\lambda), \lambda \in F\mathcal B_\Lambda \right \} \\
&= \inf \{ w_{y,\theta}(\lambda)\mid  x,z \in Z(\lambda), \lambda \in F\mathcal B_\Lambda\}
\end{split}
\label{eq:distance-from-weight}
\end{equation}
 is an ultrametric on $\Lambda^\infty$.  Furthermore, this ultrametric induces the cylinder set topology on $\Lambda^\infty$.
\end{prop}
\begin{proof}
For the first statement, combine Proposition \ref{prop:weight-gives-ultrametric} and Proposition \ref{prop:weight} or Proposition \ref{prop:weight-one-vertex}.

For the second statement,
 we will prove that if $w_{y,\theta}(\gamma) = r$ then for any $x \in Z(\gamma)$, we have $\overline{B(x; r)} = Z(\gamma)$.
Observe first that
 \begin{align*}
\overline{B(x; r)} &= \{ z \in \Lambda^\infty \mid d_{y,\theta}(x,z) \leq r \} \\
& = \left\{ z \in \Lambda^\infty\mid  \ \exists \ \lambda \in F\mathcal B_\Lambda \text{ s.t. } x, z \in Z(\lambda) \text{ and } w_{y,\theta}(\lambda) \leq r \right\} \\
& \supseteq Z(\gamma).
\end{align*}
To see that $\overline{B(x; r)} = Z(\gamma)$, choose $z \in \overline{B(x; r)}$; we will show that $z \in Z(\gamma)$.

Write $\lambda$ for the longest path in $F\mathcal B_\Lambda$ such that $x, z \in Z(\lambda)$; then $d_{y,\theta}(x,z) = w_{y,\theta}(\lambda)$.  By hypothesis,
\[w_{y,\theta}(\lambda) \leq r = w_{y,\theta}(\gamma).\]
  Moreover, since $x \in Z(\lambda) \cap Z(\gamma)$ we must have that one of $\lambda, \gamma$ is a sub-path of the other.  If $\gamma$ is a sub-path of $\lambda$ then the fact that $z \in Z(\lambda)$ implies that $z \in Z(\gamma)$ and the proof is finished.

On the other hand, if $\lambda$ is a sub-path of $\gamma$,
condition (iii) of Definition \ref{def:weight}  forces  $w_{y,\theta}(\lambda) \geq w_{y,\theta}(\gamma)$.  Thus,
\begin{equation}
\label{eq:weights-equal}
w_{y,\theta}(\lambda) \leq w_{y,\theta}(\gamma) \leq w_{y,\theta}(\lambda) \quad \Rightarrow \quad w_{y,\theta}(\lambda) = w_{y,\theta}(\gamma) = r.\end{equation}
By Lemma \ref{lem-weights-add-up}, since $\lambda$ is a sub-path of $\gamma$, we know that
\[ w_{y,\theta}(\lambda)^\theta = \sum_{\lambda \eta \in F^{|\gamma|} \mathcal B_\Lambda} w_{y,\theta}(\lambda \eta)^\theta = w_{y,\theta}(\gamma)^\theta + \sum_{\lambda \eta \not= \gamma} w_{y,\theta}(\lambda \eta)^\theta.\]
Equation \eqref{eq:weights-equal} now implies that $ \sum_{\lambda \eta \not= \gamma} w_{y,\theta}(\lambda \eta)^\theta=0$; since $w_{y,\theta}(\nu) > 0$ for any path $\nu \in F\mathcal B_\Lambda$, it follows that
\[ \{ \lambda \eta \in F^{|\gamma|} \mathcal B_\Lambda \mid  \lambda \eta \not= \gamma\} = \emptyset.\]
In other words, $Z(\lambda) = Z(\gamma)$, so $z \in Z(\gamma)$ as desired.
\end{proof}

\begin{cor}\label{cor-diam-eq-weight}
Suppose that $\Lambda$ is a finite, strongly connected $k$-graph.  Choose  $\theta \in \R_+$ and a $\R_+$-functor $y$ on $\Lambda$   such that condition (\text{w}-I) holds if $\vert \Lambda^0\vert=1$ and otherwise  condition (\text{w}-II) holds.    Then for any path $\lambda \in F\mathcal B_\Lambda$,
\[ \text{diam}\, Z(\lambda) = w_{y,\theta}(\lambda).\]
\end{cor}

\begin{proof}The proof of Proposition \ref{prop:ultrametric} establishes that for any $x \in Z(\lambda)$,
\[Z(\lambda) = \overline{B(x; r)}\]
 where $r = w_{y,\theta}(\lambda)$.  Therefore $\text{diam}\, Z(\lambda) = r = w_{y,\theta}(\lambda)$.
\end{proof}

\subsection{The Hausdorff structure of $(\Lambda^\infty, d_{y,\theta})$}
\label{sec:Hausdorff-subsection}
We conclude the paper by computing the Hausdorff measure and dimension of the ultrametric Cantor sets $(\Lambda^\infty, d_{y,\theta})$, using the detailed understanding of the weights $d_{y,\theta}$ obtained in Section \ref{sec:Bratteli-defns}, and showing that the Hausdorff measure is precisely the quasi-invariant measure $\mu_{y,\theta}$.   For this computation, the condition \eqref{eq:self similar weight} satisfied by $w_{y,\theta}$ is crucial; indeed, we are able to compute the Hausdorff measure and dimension of $(X_{\mathcal B}, d_\epsilon)$ whenever $\epsilon$ is a weight on the Bratteli diagram $\mathcal B$ which satisfies Equation \eqref{eq:self similar weight}.  To formalize this, we  introduce (see Definition \ref{def-Cantorian-Brat-diagr}) the notion of an {\em exponentially self-similar weight with exponent $\theta$} that is modeled on Equation \eqref{eq:self similar weight}. We mention that self-similar conditions on weighted Cantor sets or Bratteli diagrams have been introduced before, see for example \cite[Definition 2.6]{julien-savinien-selfsim}. The existing definitions do not apply to our case of interest, however (see Remark \ref{rmk:self-sim} below).
\begin{defn}\cite[Definition 16]{rogers}
\label{def:hausdorff-measure}
Let $(X, d)$ be a metric space and fix $s \in \R_{\geq 0}$.
The \emph{Hausdorff measure} of dimension $s$ of a  compact subset $Z$ of $X$ is
\[ H^s(Z) = \lim_{\delta \to 0}\,  \inf \left\{ \sum_{U_i \in F} (\text{diam } U_i)^s \mid |F| < \infty, \ \cup_i U_i = Z, \ \text{diam } U_i < \delta \text{ for all } i\right\}.\]
It is standard  to show that $H^s(Z)$ is a decreasing function of $s$, and that there is a unique $s \in \R$ such that $H^t(X) = \infty$ for all $t < s$ and that $H^t(X) = 0$ for all $t > s$.  This value of $s$ is called the \emph{Hausdorff dimension} of $X$.
\end{defn}

\begin{defn}
\label{def-Cantorian-Brat-diagr} Let $\mathcal{B}$ be a Bratteli diagram and let  $X_{\mathcal{B}}$ be its infinite path space equipped with the cylinder set topology.
\begin{enumerate}
\item[(i)] We say $\mathcal{B}$ is a {\em Cantorian Bratteli diagram} if $X_{\mathcal{B}}$ is a  Cantor set.
\item[(ii)] An {\em exponentially self-similar weight} with exponent $\theta\in[0,\infty)$  is a weight $\epsilon $ on a Cantorian Bratteli diagram $\mathcal{B}$  satisfying the equation
 	\begin{equation}\label{eq-self-sim-weights}
 	\epsilon (\lambda)^\theta  = \sum_{\lambda \eta \in F^{|\lambda| + m}\mathcal B} \epsilon (\lambda \eta)^\theta\end{equation}
for all finite paths $\lambda \in F\mathcal B$  and all $m\in \N$. 	\end{enumerate}
 	\end{defn}
 	
\begin{rmk}
\label{rmk:self-sim}
There does not appear to be any relation between our exponentially self-similar weights and the self-similar metrics on Bratteli diagrams discussed in \cite{julien-savinien-selfsim}.
 Lemma \ref{lem-weights-add-up} implies that for any  $\R_+$-functor $y$ and any $\theta \in \R_+$, the weights $w_{y,\theta}$ are exponentially self-similar weights with exponent $\theta$.  However, the associated metric $d_{y,\theta}$ does not satisfy the self-similarity condition in \cite[Definition 2.6]{julien-savinien-selfsim}, because $\inf \{ e^{-y(\lambda)}: \lambda \in F\mathcal B_\Lambda\} = 0$.  Moreover, the conditions placed on the constants $\{ a_\gamma: \gamma \in F\mathcal B\}$  in \cite[Definition 2.6]{julien-savinien-selfsim} are too weak to guarantee Equation \eqref{eq-self-sim-weights}.
\end{rmk}

 The proofs of the following Proposition and Corollary are identical to the proofs of Proposition \ref{prop:ultrametric} and Corollary \ref{cor-diam-eq-weight}, so we omit the details.
 \begin{prop}
\label{prop:ultrametric-Pearson-Belissard-self-sim}
  Suppose that $\mathcal{B}$ is a  Cantorian Bratteli diagram with weight  $\epsilon$. The  formula
  \begin{equation}
  \label{def-ultrametric-from-weight}
  d_\epsilon(x,y ) := \epsilon(x\wedge y)
  \end{equation}
 defines  an ultrametric on the infinite path space  $X_{\mathcal{B}}$ of $\mathcal{B}$.  Furthermore, if $\epsilon$ is exponentially self-similar with exponent $\theta$,
 then the ultrametric $d_\epsilon$ induces the cylinder set topology on $X_{\mathcal B}$.
\end{prop}

 \begin{cor}\label{cor-diam-eq-weight-self-sim}
  Suppose that $\mathcal{B}$ is a  Cantorian Bratteli diagram with exponentially self-similar weight  $\epsilon$. Then in $(X_{\mathcal{B}}, d_{\epsilon})$ we have
 \[ \text{diam}\, Z(\lambda) = \epsilon(\lambda).\]
 \end{cor}

 We now prove that for Cantorian Bratteli diagrams with exponentially self-similar weights, the exponent and weight are intimately tied to the Hausdorff dimension and measure of the ultrametric space $(X_{\mathcal B}, d_\epsilon)$;   compare the results below with \cite[pages 203-206]{edgar-meas}.

 \begin{thm}
 	\label{thm:Hausdorff-dim-self-sim}
 Suppose that $\mathcal{B}$ is a weighted Cantorian Bratteli diagram for a self-similar weight  $\epsilon$ with exponent $\theta$.
 Then  the Hausdorff dimension of the ultrametric Cantor set $(X_{\mathcal{B}}, d_{\epsilon})$  is $\theta$.
 \end{thm}

 \begin{proof}
 We first compute the Hausdorff measure of dimension $\theta$  of $(X_{\mathcal{B}}, d_{\epsilon})$.
  Consider  a finite  cover   $\{ U_i\}_{i=1,\ldots, n}$ of $X_{\mathcal B}$, and the  associated sum
  \[
   \sum_{i=1}^n (\text{diam } U_i)^\theta.
  \]
  Observe first that  we can assume each $U_i$ to be a cylinder set $Z(\mu_i)$.
   Indeed, given   $U_i$, we can  pick $x = x_1 x_2 \cdots  \in U_i$, and define $\overline{B_{x,i}}$ to be the closed
   ball of center $x$ and radius $\text{diam } U_i$, namely,  
   \[\overline{B_{x,i}}= \{ y \in X_{\mathcal B}  \mid d_\epsilon (x, y) \leq  \text{diam } U_i\}.\]
     We will show that $\overline{B_{x,i} }= Z(x_1 \cdots x_{\ell_i})$ for some $\ell_i \in \N$; that $\text{diam}\, \overline{B_{x, i} }\leq \text{diam}\, U_i$; and that $\overline{B_{x,i}} \supseteq U_i$.  Thus, in order to minimize the sum used to compute the Hausdorff measure, we may assume without loss of generality that each open set  $U_i $ is a cylinder set, by replacing  $U_i$ with $ \overline{B_{x,i}} = Z(x_1 \cdots x_{\ell_i})$.

  The fact that $\overline{B_{x, i}} \supseteq U_i$ is immediate from the observation that $ d_\epsilon (x, z) \leq \text{diam}\, U_i $ for all $z \in U_i$.  To estimate the diameter of $\overline{B_{x,i}}$, choose $y, z \in \overline{B_{x,i}}$ and observe that
  \[ d_\epsilon(y, z) \leq \max \{ d_\epsilon(x, y) , d_\epsilon(x, z) \} \leq \text{diam}\, U_i.\]
  Taking supremums reveals that $\text{diam}\, \overline{B_{x, i}} \leq \text{diam}\, U_i$.

  We now check that $\overline{B_{x,i} }= Z(x_1 \cdots x_{\ell_i})$. {By  the definition of the weight $\epsilon$, there is a smallest $ \ell_i \in \N$ such that
  \[\epsilon (x_1 \cdots x_{\ell_i}) \leq \text{diam}\, \overline{B_{x,i}}.\]
 If $y \in \overline{B_{x,i}}$, by Equation \eqref{def-ultrametric-from-weight} we have that
 \[ \text{diam}\,\overline{B_{x,i}} \geq d_\epsilon(x, y ) = w_\epsilon(x \wedge y ) = \epsilon(x_1 \cdots x_{m_i})\]
 for some $\N \ni m_i \geq \ell_i $ (thanks to Definition \ref{def:weight} and
 the minimality of $\ell_i$).  It follows that $y \in Z(x_1 \cdots x_{\ell_i})$.
 On the other hand, if $z\in Z(x_1 \cdots x_{\ell_i}),$ then
 \[ d_\epsilon(z, x) = \epsilon(z \wedge x) \leq \epsilon(x_1 \cdots x_{\ell_i}) \leq \text{diam}\, U_i,\]
 so $z \in \overline{B_{x, i}}$ by construction.  In other words, $\overline{B_{x, i}} = Z(x_1 \cdots x_{\ell_i})$ as claimed; set $\mu_i :=x_1 \cdots x_{\ell_i} ,$ $i=1,\dots, n$.}

 Thus, in estimating $H^\theta(X_{\mathcal B})$, we need to compute
 \begin{equation} \inf \left\{ \sum_{i=1}^n (\text{diam}\, Z(\mu_i))^\theta \mid \bigcup_i Z(\mu_i) = X_{\mathcal B}, \ \text{diam}\, Z(\mu_i) < \delta  \text{ for all } i\right\}
 \label{eq:Hausdorff-estimate}
 \end{equation}
 and take the limit as $\delta \to 0$ of these infima.

Given one cover $\mathcal U = \{ Z(\mu_i)\}_{i=1}^n$ of $X_{\mathcal B}$ in the set \eqref{eq:Hausdorff-estimate}, let $M$ be the maximum of the lengths of the paths $\mu_1, \ldots,\mu_n$.  Moreover,  by  equation \eqref{eq-self-sim-weights},  replacing each $Z(\mu_i)$ by the collection 
\[\{ Z(\mu_i \eta_{ij}) \mid  |\eta_{ij}| = M - |\mu_i|\}_j\]
 (which makes all of the cylinder sets in the open cover $\mathcal{U}$ of the same length) does not change the sum arising in our computation \eqref{eq:Hausdorff-estimate}.  (Note that $Z(\mu_i) = \bigsqcup_j Z(\mu_i \eta_{ij})$, so the new collection of sets does indeed cover $X_{\mathcal{B}}$ whenever  $\{ Z(\mu_i)\}_{i=1 }^n$ does.)

 Therefore we can assume that all of the cylinder sets $ Z(\mu_1),\ldots,Z(\mu_n )$ are associated to finite paths in $\mathcal B$ which all have the same length $M$, and that all of the cylinder sets are also pairwise disjoint, since
 \[|\mu_i| = |\mu_j| \Rightarrow Z(\mu_i) \cap Z(\mu_j) = \delta_{i,j} Z(\mu_j).\]
 Since $\{ Z(\mu_i)\}_i$  covers $X_{\mathcal B}$, the collection $\mathcal{U}= \{ Z(\mu_1),\ldots,Z(\mu_n )\}$
 must therefore contain precisely all of  the cylinder sets of length $M$ in $X_{\mathcal B}$.

 Now by applying  Corollary \ref{cor-diam-eq-weight-self-sim}, we see that
 \begin{align*}
 H^\theta(X_{\mathcal B})& = \lim_{\delta \to 0 } \inf \left\{ \sum_{Z(\mu_i) \in \mathcal U} (\text{diam}\, Z(\mu_i))^\theta \mid  \text{diam}\, Z(\mu_i) < \delta  \ \forall \ i\right\} \\
 &= \lim_{\delta \to 0} \inf \left\{ \sum_{|\mu| = M}(\epsilon (\mu))^\theta \mid \max \{ \epsilon(\mu)\mid  |\mu| = M\} < \delta \right\} \\
 &= \lim_{\delta \to 0}  \inf \left\{ \sum_{v\in \Lambda^0} \epsilon(v)^\theta \right\} ,\\
 \end{align*}
 which is strictly between 0 and infinity.
 Definition \ref{def:hausdorff-measure} now implies that $\theta$ is the Hausdorff dimension of $(X_{\mathcal B}, d_{\epsilon})$.  Moreover, the arguments above imply that the Hausdorff measure $H^\theta$ is given on cylinder sets by
 \[ H^\theta(Z(\lambda)) = \epsilon(\lambda)^\theta.\]
 \end{proof}

\begin{cor}
	\label{cor:Hausdorff-dim}
	Let $\Lambda$ be a finite, strongly connected $k$-graph,   $y$ a  $\R_+$-functor  on $\Lambda$ and   $\theta \in (0,\infty)$ such that condition (\text{w}-I) holds if $\vert \Lambda^0\vert=1$ and otherwise  condition (\text{w}-II) holds.    Then  the Hausdorff dimension of the ultrametric Cantor set $(\Lambda^\infty, d_{y, \theta})$  is $\theta$.  Moreover, the Hausdorff measure $H^\theta$ agrees with $\mu_{y,\theta}$.
\end{cor}

\begin{proof} Since, by Lemma~\ref{lem-weights-add-up}, $w_{y,\theta}$ is a self-similar weight on $\Lambda^\infty\cong X_{\mathcal{B}_\Lambda}$ with exponent $\theta$, the assertion about Hausdorff dimension follows directly from Theorem~\ref{thm:Hausdorff-dim-self-sim}.
To see that $H^\theta = \mu_{y,\theta}$, 
observe that
\begin{equation}
\label{eq:hausdorff-equals-weight}
 H^\theta(Z(\lambda))= w_{y,\theta}(\lambda)^\theta = e^{-\theta y(\lambda)} \rho(B(y,\theta))^{-d(\lambda)} \xi^{y,\theta}_{s(\lambda)} = \mu_{y,\theta}(Z(\lambda)).\end{equation}
Since the cylinder sets generate the ultrametric topology on $\Lambda^\infty$ by Proposition \ref{prop:ultrametric}, $\mu_{y,\theta} = H^\theta$ as claimed.
\end{proof}

\bibliographystyle{amsalpha}
\bibliography{eagbib}

\end{document}